\documentclass[final,leqno]{amsart}

\usepackage[T1]{fontenc}
\usepackage[utf8]{inputenc}
\usepackage[american]{babel}

\usepackage{amsmath}
\usepackage{amsthm}
\usepackage{amssymb}

\theoremstyle{remark}
  \newtheorem{rem}{Remark}
\theoremstyle{plain}
  \newtheorem{lem}{Lemma}
  \newtheorem{thm}{Theorem}

\newcommand*\qedspace{\nobreak\hspace{2em minus 2em}}

\usepackage[scr=boondox]{mathalfa}
\DeclareFontFamily{OT1}{pzc}{}
\DeclareFontShape{OT1}{pzc}{m}{it}{ <-> s * [1.125] pzcmi7t}{}
\DeclareMathAlphabet{\mathpzc}{OT1}{pzc}{m}{it}

\newcommand*\hairspace{\kern 0.08333em}
\newcommand*\thinglue{\nobreak\hspace{.16667em plus .08333em}}

\newcommand*\wc{{\mkern 2mu\cdot\mkern 2mu}}

\let\op\mathcal
\let\form\mathpzc
\let\vect\mathpzc

\providecommand*\mathllap[1]{{}\llap{$\displaystyle#1$}}

\DeclareMathOperator*{\esssup}{ess\,sup}
\let\Im\undefined\DeclareMathOperator{\Im}{Im}
\let\Re\undefined\DeclareMathOperator{\Re}{Re}

\DeclareMathOperator{\Lip}{Lip}

\providecommand{\rbr}[1]{(#1)}
\providecommand{\bigrbr}[1]{\bigl(#1\bigr)}

\providecommand{\biggrbr}[1]{\biggl(#1\biggr)}

\providecommand{\sbr}[1]{[#1]}

\providecommand{\cbr}[1]{\{#1\}}
\providecommand{\bigcbr}[1]{\bigl\{#1\bigr\}}
\providecommand{\Bigcbr}[1]{\Bigl\{#1\Bigr\}}

\providecommand{\abr}[1]{\langle#1\rangle}
\providecommand{\bigabr}[1]{\bigl\langle#1\bigr\rangle}

\providecommand{\abs}[1]{\lvert#1\rvert}
\providecommand{\bigabs}[1]{\bigl\lvert#1\bigr\rvert}

\providecommand{\biggabs}[1]{\biggl\lvert#1\biggr\rvert}

\providecommand{\norm}[1]{\lVert#1\rVert}
\providecommand{\bignorm}[1]{\bigl\lVert#1\bigr\rVert}

\def\B{\mathbf{B}}
\def\M{\mathbf{M}}
\def\S{\mathbf{S}}
\let\set\mathscr
\def\C{\mathbb{C}}
\def\N{\mathbb{N}}
\def\R{\mathbb{R}}
\def\T{\mathbb{T}}
\def\Z{\mathbb{Z}}
\let\Natto\sbr

\usepackage{graphicx}

\title[HOMOGENIZATION FOR OPERATORS ON AN INFINITE CYLINDER]%
      {HOMOGENIZATION FOR NON-SELF-ADJOINT PERIODIC ELLIPTIC OPERATORS
       ON AN INFINITE CYLINDER}%
       \thanks{The~author was supported by SPbSU grant 0.38.237.2014
               and by RFBR grant 14-01-00760.}%
\author{Nikita N. Senik}%
        \thanks{St.~Petersburg State University,
                Universitetskaya nab.~7/9, St.~Petersburg 199034, Russia
                (\protect\raisebox{-1.75pt}{\protect\includegraphics{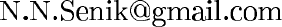}}).%
               }%

\begin{document}

\begin{abstract}
We consider the problem of homogenization for non-self-adjoint second-order
elliptic differential operators~$\mathcal{A}^{\varepsilon}$ of divergence
form on $L_{2}(\mathbb{R}^{d_{1}}\times\mathbb{T}^{d_{2}})$, where
$d_{1}$ is positive and~$d_{2}$ is non-negative. The~coefficients
of the operator~$\mathcal{A}^{\varepsilon}$ are  periodic in the
first variable with period~$\varepsilon$ and smooth in a certain
sense in the second.  We show that, as $\varepsilon$ gets small,
$(\mathcal{A}^{\varepsilon}-\mu)^{-1}$ and~$D_{x_{2}}(\mathcal{A}^{\varepsilon}-\mu)^{-1}$
converge in the operator norm  to, respectively, $(\mathcal{A}^{0}-\mu)^{-1}$
and~$D_{x_{2}}(\mathcal{A}^{0}-\mu)^{-1}$, where $\mathcal{A}^{0}$
is an operator whose coefficients depend only on~$x_{2}$. We also
obtain an approximation for $D_{x_{1}}(\mathcal{A}^{\varepsilon}-\mu)^{-1}$
and find the next term in the approximation for~$(\mathcal{A}^{\varepsilon}-\mu)^{-1}$.
Estimates for the rates of convergence and the rates of approximation
are provided and are sharp with respect to the order.
\end{abstract}



\maketitle

\section{Introduction}

The~periodic homogenization problem consists in studying asymptotic
behavior of solutions of  differential equations with rapidly oscillating
coefficients. The~key fact about homogenization is that the solutions
of such problems converge to solutions of problems whose coefficients
no longer oscillate; in applications, this means that we approximate
highly heterogeneous media by a homogenous one. Classical arguments
(as in \cite{BLP}, \cite{BP} or~\cite{ZhKO}) show that the convergence
is weak or strong. In~certain cases, they may even imply  the norm-resolvent
convergence of the corresponding operators  (see~\cite{OShY});
however, the question of whether these operators converge in the norm-resolvent
sense without any additional regularity assumptions remained open
until 2001, when \textsc{Birman} and~\textsc{Suslina~}\cite{BSu1}
(see also~\cite{BSu2}) proved that this is the case for a broad
class of  elliptic problems. Many related results appeared in subsequent
years; see, for example, \cite{Gr1}, \cite{Gr2}, \cite{Zh}, \cite{ZhP},
\cite{Bo}, \cite{KLS} and references therein. In~the~recent paper~\cite{ChC},
a result of this kind was obtained for some of the high-contrast problems.

The~present paper is motivated by the study of homogenization problems
for operators whose coefficients are periodic only in certain directions.
These arise naturally in many applications~-- for instance, in the
theory of waveguides and in elasticity,~-- and were investigated
in \cite{S-HT}, \cite{OShY},  \cite{Su1}, \cite{BCSu} and~\cite{Se1}.

Let $\Xi$ be the cylinder~$\mathbb{R}^{d_{1}}\times\mathbb{T}^{d_{2}}$.
In~\cite{Su1}, \textsc{Suslina} studied the homogenization problem
for elliptic self-adjoint operators~$\op A^{\varepsilon}$ on $\Xi$
of the form 
\[
\op A^{\varepsilon}=D_{1}^{*}A_{11}\rbr{\varepsilon^{-1}x_{1},x_{2}}D_{1}+D_{2}^{*}A_{22}\rbr{\varepsilon^{-1}x_{1},x_{2}}D_{2}.
\]
Here, $A_{11}$ and~$A_{22}$ are periodic in the first variable
and Lipschitz in the second. She proved that $\op A^{\varepsilon}$
converges in the norm-resolvent sense to an operator~$\op A^{0}$,
whose coefficients depend only on the non-periodic variable~$x_{2}$,
and furthermore
\[
\norm{\rbr{\op A^{\varepsilon}+1}^{-1}-\rbr{\op A^{0}+1}^{-1}}_{\B\rbr{L_{2}\rbr{\Xi}}}\le C\varepsilon.
\]
Such problems were further analyzed in~\cite{Se1}, where we extended
that result to self-adjoint operators with lower-order terms and also
obtained an approximation for the resolvent in~$\B\rbr{L_{2}\rbr{\Xi},H^{1}\rbr{\Xi}}$.
(Strictly speaking, the work~\cite{Se1} deals with the case~$d_{1}=d_{2}=1$,
although it is possible to use the techniques of that article to treat
the other cases.) But operators with non-diagonal terms were left
beyond the scope of these papers, and it is our intention here to
fill this gap.

In~this article, we will be concerned with an elliptic non-self-adjoint
operator~$\op A^{\varepsilon}$ on $\Xi$ of the form
\[
\op A^{\varepsilon}=D^{*}A\rbr{\varepsilon^{-1}x_{1},x_{2}}D+D^{*}a_{1}\rbr{\varepsilon^{-1}x_{1},x_{2}}+a_{2}^{*}\rbr{\varepsilon^{-1}x_{1},x_{2}}D+q\rbr{\varepsilon^{-1}x_{1},x_{2}},
\]
where $A$, $a_{1}$, $a_{2}$ and~$q$ are periodic in the first
variable with respect to a lattice in $\R^{d_{1}}$ and have weak
derivatives with respect to the second variable. We further assume
that the coefficients  together with the derivatives belong to certain
spaces of Sobolev multipliers.   We find approximations for $\rbr{\op A^{\varepsilon}-\mu}^{-1}$
and~$D\rbr{\op A^{\varepsilon}-\mu}^{-1}$ in the operator norm and
prove the following estimates:
\begin{gather}
\norm{\rbr{\op A^{\varepsilon}-\mu}^{-1}-\rbr{\op A^{0}-\mu}^{-1}}_{\B\rbr{L_{2}\rbr{\Xi}}}\le C\varepsilon,\label{est: Intro - convergence for A=00207B=0000B9}\\
\norm{D_{2}\rbr{\op A^{\varepsilon}-\mu}^{-1}-D_{2}\rbr{\op A^{0}-\mu}^{-1}}_{\B\rbr{L_{2}\rbr{\Xi}}^{d_{2}}}\le C\varepsilon\label{est: Intro - convergence for D=002082A=00207B=0000B9}
\end{gather}
and
\begin{gather}
\norm{D_{1}\rbr{\op A^{\varepsilon}-\mu}^{-1}-D_{1}\rbr{\op A^{0}-\mu}^{-1}-\varepsilon D_{1}\op K_{\mu}^{\varepsilon}}_{\B\rbr{L_{2}\rbr{\Xi}}^{d_{1}}}\le C\varepsilon,\label{est: Intro - approximation for D=002081A=00207B=0000B9}\\
\norm{\rbr{\op A^{\varepsilon}-\mu}^{-1}-\rbr{\op A^{0}-\mu}^{-1}-\varepsilon\op C_{\mu}^{\varepsilon}}_{\B\rbr{L_{2}\rbr{\Xi}}}\le C\varepsilon^{2}\label{est: Intro - approximation for A=00207B=0000B9}
\end{gather}
(see~the~statements of Theorems~\ref{thm: Convergence}\hairspace--\hairspace\ref{thm: Approximation with 2d corrector}
in Section~\ref{sec: Main results}). Here, $\op A^{0}$ is the effective
operator and $\op K_{\mu}^{\varepsilon}$ and~$\op C_{\mu}^{\varepsilon}$
are correctors. The~effective operator has a form similar to that
of~$\op A^{\varepsilon}$, with coefficients depending only on the
non-periodic variable~$x_{2}$, while the correctors involve rapidly
oscillating functions. The~estimates are sharp with respect to the
order and the constants on the right may be expressed explicitly in
terms of the problem parameters. Some of our results were announced
in~\cite{Se2}.

The~estimates~(\ref{est: Intro - convergence for D=002082A=00207B=0000B9})
and~(\ref{est: Intro - approximation for A=00207B=0000B9}) have
no analogue in \cite{Su1} and~\cite{Se1}. The~former is new and
the latter has appeared for the first time in \cite{BSu3} for certain
self-adjoint operators on the entire space. A~more recent development~\cite{Su4}
has extended that result to self-adjoint operators with lower-order
terms. We also mention here the paper~\cite{P}, where an estimate
similar to (\ref{est: Intro - approximation for A=00207B=0000B9})
was obtained for  operators on $\R^{d}$ with smooth coefficients.
However, all those results apply only to purely periodic operators.

On~the~other hand, we may regard the problem we address here as
a special case of more general  locally periodic homogenization problems
(where the coefficients may depend on both $x$ and~$\varepsilon^{-1}x$).
From~this point of view, estimates of type (\ref{est: Intro - convergence for A=00207B=0000B9})
and~(\ref{est: Intro - approximation for D=002081A=00207B=0000B9})
are known; see, for instance, \cite{PT}, where symmetric operators
with no lower-order terms were treated. In~contrast, the estimate~(\ref{est: Intro - convergence for D=002082A=00207B=0000B9})
is a feature of our problem. As~for~(\ref{est: Intro - approximation for A=00207B=0000B9}),
we believe that the arguments provided here can be used to prove a
similar result for locally periodic operators as well.

 The~operator-theoretic method of \textsc{Birman} and~\textsc{Suslina}
deals only with purely periodic operators and cannot be extended to
locally-periodic ones. Nevertheless,  the abstract results  
they obtained may be adapted, by ad hoc means, to  get the approximations
for operators with $A$ having block-diagonal structure, as shown
in~\cite{Su1} and~\cite{Se1}. However, operators with more general~$A$
do not fit into this framework. So,  if we are to handle these
cases, we must develop a  different approach.

Our program is as follows. We first reduce the problem to a problem
on a fundamental domain for the lattice. This is done by applying
the scaling transformation and the Gelfand transform, both with respect
to the periodic variable. This step is identical to the one in \cite{Su1}
and~\cite{Se1}. The~next step differs significantly. We prove
suitable versions of the resolvent identity (see~(\ref{eq: Identity for U(=0003C4)})
and~(\ref{eq: Identity for V(=0003C4)})).  This enables us to
verify the desired inequalities by elementary means.

Note that the torus~$\mathbb{T}^{d_{2}}$ can be replaced by any
flat manifold without boundary ($\R^{d_{2}}$, for~instance). We
hope that the techniques presented in this article will prove useful
in studying homogenization problems  on domains of type~$\R^{d_{1}}\times\rbr{0,1}^{d_{2}}$
with Dirichlet or Neumann boundary conditions as well.

The~paper is organized as follows. In~Section~\ref{sec: Problem formulation},
we give the necessary background information, introduce the operator~$\op A^{\varepsilon}$,
as well as the effective operator and the correctors,  formulate
the problem under consideration and provide an example of~$\op A^{\varepsilon}$.
Section~\ref{sec: Main results} contains presentation of the main
results. In~Section~\ref{sec: Problem on the fundamental domain},
we deal with the problem on the fundamental domain and prove the results.

\section{\label{sec: Problem formulation}Basic definitions and problem formulation}

\hspace{0em minus 0.5em}We begin with some notation.

\subsection{Preliminaries}

The~symbol~$\norm{\wc}_{U}$ denotes the norm on a normed space~$U$.
Let $U$ and~$V$ be Banach spaces. We use the notation~$\B\rbr{U,V}$
to denote the Banach space of bounded linear operators from $U$ to~$V$.
When~$U=V$, the space~$\B\rbr U=\B\rbr{U,U}$ becomes a Banach
algebra with identity~$\op I$. The~inner product on a pre-Hilbert
space~$U$ is denoted by~$\rbr{\wc,\wc}_{U}$. In~the~finite-dimensional
case~$U=\C^{n}$, the norm and the inner product are denoted by $\abs{\wc}$
and~$\abr{\wc,\wc}$, respectively. We shall identify the spaces~$\B\rbr{\C^{n},\C^{m}}$
 with~$\C^{m\times n}$.

Let $\Sigma$ be a  domain in $\R^{d}$ and $U$ a Banach space.
Then $L_{p}\rbr{\Sigma;U}$, with $1\le p\le\infty$, is the Banach
space of strongly measurable functions~$u\colon\Sigma\to U$ satisfying

\[
\norm u_{L_{p}\rbr{\Sigma;U}}=\biggrbr{\int_{\Sigma}\norm{u\rbr x}_{U}^{p}\,dx}^{1/p}<\infty
\]
if~$p<\infty$ and
\[
\norm u_{L_{\infty}\rbr{\Sigma;U}}=\esssup_{x\in\Sigma}\norm{u\rbr x}_{U}<\infty
\]
if~$p=\infty$. In~case~$U=\C^{n}$, we shall write $\norm{\wc}_{p,\Sigma}$
for the norm on~$L_{p}\rbr{\Sigma}^{n}=L_{p}\rbr{\Sigma;U}$ and~$\rbr{\wc,\wc}_{2,\Sigma}$
for the inner product on~$L_{2}\rbr{\Sigma}^{n}=L_{2}\rbr{\Sigma;U}$.
We denote by $W_{p}^{m}\rbr{\Sigma}$, with $m\in\N$ and~$1\le p\le\infty$,
the Banach space of those measurable functions~$u\colon\Sigma\to\C$
that possess all weak derivatives up to and including order~$m$
and such that

\[
\norm u_{m,p,\Sigma}=\norm u_{W_{p}^{m}\rbr{\Sigma}}=\biggrbr{\sum_{\abs{\alpha}\le m}\norm{D^{\alpha}u}_{L_{p}\rbr{\Sigma}}^{p}}^{1/p}<\infty
\]
if~$p<\infty$ and
\[
\norm u_{m,\infty,\Sigma}=\norm u_{W_{\infty}^{m}\rbr{\Sigma}}=\max_{\abs{\alpha}\le m}\norm{D^{\alpha}u}_{L_{\infty}\rbr{\Sigma}}<\infty
\]
if~$p=\infty$. The~Hilbert space~$W_{2}^{m}\rbr{\Sigma}$ is
denoted by~$H^{m}\rbr{\Sigma}$, and $H^{m}\rbr{\Sigma}^{*}$ is
its dual space under the pairing~$\rbr{\wc,\wc}_{2,\Sigma}$.
If $\Sigma$ is not open, then $W_{p}^{m}\rbr{\Sigma}$ will be understood
to mean the Sobolev space on the interior of~$\Sigma$.

Multipliers between Sobolev spaces are (generalized) functions such
that the corresponding multiplication operators are bounded. Here
we shall be brief; a thorough treatment of Sobolev multipliers may
be found in~\cite{MSh}. Let $\Sigma$ be a Lipschitz domain in $\R^{d}$,
and let $m$ and~$n$ be non-negative integers satisfying~$m\ge n$.
Then $\gamma$ is a Sobolev multiplier between $H^{m}\rbr{\Sigma}$
and~$H^{n}\rbr{\Sigma}$ (written~$\gamma\in\M\rbr{H^{m}\rbr{\Sigma},H^{n}\rbr{\Sigma}}$)
provided that the operator of multiplication~$\gamma\colon H^{m}\rbr{\Sigma}\to H^{n}\rbr{\Sigma}$
is continuous. The~space~$\M\rbr{H^{m}\rbr{\Sigma},H^{n}\rbr{\Sigma}^{*}}$
of Sobolev multipliers between $H^{m}\rbr{\Sigma}$ and the dual of
$H^{n}\rbr{\Sigma}$ is defined in the same way. Notice that an element
of $\M\rbr{H^{m}\rbr{\Sigma},H^{n}\rbr{\Sigma}^{*}}$ is a complex-valued
distribution.

We shall normally write $\norm{\gamma}_{\M}$ for the norm of a Sobolev
multiplier~$\gamma$. This should lead to no confusion, since, once
we discover that $\gamma\in\M\rbr{H^{m}\rbr{\Sigma},H^{n}\rbr{\Sigma}}$
(or $\gamma\in\M\rbr{H^{m}\rbr{\Sigma},H^{n}\rbr{\Sigma}^{*}}$),
we fix the spaces~$H^{m}\rbr{\Sigma}$ and~$H^{n}\rbr{\Sigma}$
(or the spaces~$H^{m}\rbr{\Sigma}$ and~$H^{n}\rbr{\Sigma}^{*}$).

Given a positive $\delta$, the scaling transformation~$\op S^{\delta}$
is defined to be the map that assigns to each measurable function~$u$
on $\Sigma$ the measurable function~$v$ on $\delta^{-1}\Sigma$
given by~$v\rbr y=\delta^{d/2}u\rbr{\delta y}$. Then $\op S^{\delta}$
is  an isomorphism of $H^{m}\rbr{\Sigma}$ onto~$H^{m}\rbr{\delta^{-1}\Sigma}$,
with
\[
\norm{\op S^{\delta}}_{\B\rbr{H^{m}\rbr{\Sigma},H^{m}\rbr{\delta^{-1}\Sigma}}}\le\max\cbr{1,\delta^{m}},
\]
and an isometry provided that~$m=0$. By~duality $\op S^{\delta}$
extends to $H^{m}\rbr{\Sigma}^{*}$, so that $\op S^{\delta}\colon H^{m}\rbr{\Sigma}^{*}\to H^{m}\rbr{\delta^{-1}\Sigma}^{*}$
is also an isomorphism and
\[
\norm{\op S^{\delta}}_{\B\rbr{H^{m}\rbr{\Sigma}^{*},H^{m}\rbr{\delta^{-1}\Sigma}^{*}}}\le\max\cbr{1,\delta^{-m}}.
\]

Let $\cbr{\lambda_{m}}_{m\in\Natto d}$  be a basis of~$\R^{d}$.
Here, $\sbr d$ denotes the set of integers~$\cbr{1,2,\ldots,d}$.
Then the basis generates the lattice
\[
\Lambda=\Bigcbr{\lambda\in\R^{d}\colon\lambda=\sum_{m\in\Natto d}n_{m}\lambda_{m},n_{m}\in\Z}
\]
with~the~basic cell
\[
\Omega=\Bigcbr{x\in\R^{d}\colon x=\sum_{m\in\Natto d}x_{m}\lambda_{m},x_{m}\in[-2^{-1},2^{-1})}.
\]
The~dual lattice~$\Lambda^{*}$ is generated by the basis~$\cbr{\lambda_{m}^{*}}_{m\in\Natto d}$
that is defined by the equations~$\abr{\lambda_{m},\lambda_{n}^{*}}=2\pi\delta_{mn}$.
We denote the Brillouin zone by~$\Omega^{*}$:
\[
\Omega^{*}=\bigcbr{k\in\R^{d}\colon\abs k<\abs{k-\lambda^{*}},\lambda^{*}\in\Lambda^{*}\setminus\cbr 0}.
\]
Notice that the closure of $\Omega^{*}$ is a convex polyhedron containing
the ball of radius~$r_{\Lambda}=2^{-1}\min_{\lambda^{*}\in\Lambda^{*}\setminus\cbr 0}\abs{\lambda^{*}}$
centered at the origin.

Lattices are intimately related to Fourier series. If~$u$ is any
function in $L_{2}\rbr{\Omega}$, then there is a unique representation
\[
u\rbr x=\abs{\Omega}^{-1/2}\sum_{\lambda^{*}\in\Lambda^{*}}\hat{u}_{\lambda^{*}}e^{-i\abr{x,\lambda^{*}}},
\]
where the series converges in~$L_{2}\rbr{\Omega}$. The~corresponding
mapping~$u\mapsto\cbr{\hat{u}_{\lambda^{*}}}_{\lambda^{*}\in\Lambda^{*}}$
is an isometric isomorphism of $L_{2}\rbr{\Omega}$ onto~$l_{2}\rbr{\Lambda^{*}}$.

Let~$\widetilde W_{p}^{m}\rbr{\Omega}$ denote the subspace of $W_{p}^{m}\rbr{\Omega}$
consisting of functions whose periodic extensions are in~$W_{p,\mathrm{loc}}^{m}\rbr{\R^{d}}$.
The~symbol~$\widetilde W_{p,0}^{m}\rbr{\Omega}$ stands for the space of
functions in $\widetilde W_{p}^{m}\rbr{\Omega}$ with zero mean value. We
shall write $\widetilde H^{m}\rbr{\Omega}$ and~$\widetilde H_{0}^{m}\rbr{\Omega}$
for $\widetilde W_{2}^{m}\rbr{\Omega}$ and~$\widetilde W_{2,0}^{m}\rbr{\Omega}$.
Observe that, for each~$k\in\R^{d}$ and any $u\in\widetilde H^{1}\rbr{\Omega}$,
we have
\[
\norm{\rbr{D+k}u}_{2,\Omega}^{2}=\sum_{\lambda^{*}\in\Lambda^{*}}\abs{\lambda^{*}+k}^{2}\abs{\hat{u}_{\lambda^{*}}}^{2},
\]
which yields a variant of Poincar\'{e}'s inequality: 
\begin{equation}
\norm{u-u_{\Omega}}_{2,\Omega}\le C_{\Omega}\norm{\rbr{D+k}u}_{2,\Omega},\label{est: Poincar=0000E9's inequality}
\end{equation}
all~$k\in\Omega^{*}$. Here, $u_{\Omega}=\abs{\Omega}^{-1}\int_{\Omega}u\rbr x\,dx$
and~$C_{\Omega}=r_{\Lambda}^{-1}$.

Another operator that is closely related to lattices is the Gelfand
transform $\op G\colon L_{2}\rbr{\R^{d}}\to L_{2}\rbr{\Omega^{*}\times\Omega}$
given by
\[
\rbr{\op Gu}\rbr{k,x}=\abs{\Omega^{*}}^{-1/2}\sum_{\lambda\in\Lambda}u\rbr{x+\lambda}\,e^{-i\abr{x+\lambda,k}},
\]
the~series converging in~$L_{2}\rbr{\Omega^{*}\times\Omega}$. It~is
well known that $\op G$ is an isometric isomorphism of $L_{2}\rbr{\R^{d}}$
onto $L_{2}\rbr{\Omega^{*}\times\Omega}$ and an isomorphism of $H^{1}\rbr{\R^{d}}$
onto~$L_{2}\rbr{\Omega^{*};\widetilde H^{1}\rbr{\Omega}}$. By~duality,
the Gelfand transform extends to $H^{1}\rbr{\R^{d}}^{*}$, so~$\op G\colon H^{1}\rbr{\R^{d}}^{*}\to L_{2}\rbr{\Omega^{*};\widetilde H^{1}\rbr{\Omega}}^{*}$.

\subsection{Problem formulation}

We fix a positive integer~$d_{1}$ and a non-negative integer~$d_{2}$;
the first will be the number of the periodic directions, and the second
will be the number of the non-periodic directions. We suppose for
specificity that $d_{2}$ is positive; the case~$d_{2}=0$ is similar,
with obvious changes. Let~$d=d_{1}+d_{2}$. Set $\Xi=\R^{d_{1}}\times\T_{L}^{d_{2}}$,
$L>0$, where $\T_{L}^{d_{2}}$ stands for the $d_{2}$\nobreakdash-dimensional
flat torus~$\R^{d_{2}}\slash\rbr{L\Z^{d_{2}}}$; that is, a cube
in $\R^{d_{2}}$ with opposite sides identified. Now for each $x\in\Xi$,
we have $x=\rbr{x_{1},x_{2}}$, where $x_{1}\in\R^{d_{1}}$ and~$x_{2}\in\T_{L}^{d_{2}}$.
The~$m$th coordinate of $x_{1}$ and the $n$th coordinate of
$x_{2}$ are denoted by $x_{1,m}$ and~$x_{2,n}$, respectively.

Let $\Lambda$ be a lattice in $\R^{d_{1}}$ acting on~$\Xi$. If~we
denote a basic cell of $\Lambda$ by $\Omega_{1}$ and the torus~$\T_{L}^{d_{2}}$
by $\Omega_{2}$, then $\Omega=\Omega_{1}\times\Omega_{2}$ is a
fundamental domain for $\Lambda$, and $\cbr{\Omega_{\lambda}}_{\lambda\in\Lambda}$,
where $\Omega_{\lambda}=\lambda+\Omega$, is a tiling of~$\Xi$.

We now introduce a class of allowed coefficients.  Let $U$ and~$V$
be complex Sobolev spaces over the interior of $\Omega$ or subspaces
of such spaces. We define $\S\rbr{U,V}$ to be the set of all complex-valued
generalized functions~$\gamma\in C_{0}^{\infty}\rbr{\Xi}^{*}$ such
that (1)~$\gamma$ is periodic with respect to~$\Lambda$, (2)~$\gamma\in\M\rbr{U,V}$,
and~(3)~$D_{x_{2}}\gamma\in\M\rbr{U,V}^{d_{2}}$. We shall write
$\S\rbr U$ in place of~$\S\rbr{U,U}$.

Let $A$ be a matrix-valued function in $\S\rbr{L_{2}\rbr{\Omega}}^{d\times d}$
with $\Re A$ uniformly positive definite, $a_{1}$ and~$a_{2}$
be vector-valued functions in $\S\rbr{H^{1}\rbr{\Omega},L_{2}\rbr{\Omega}}^{d}$,
and $q$ be a complex-valued distribution in~$\S\rbr{H^{1}\rbr{\Omega},H^{1}\rbr{\Omega}^{*}}$.
Assume also that 
\begin{equation}
\norm{a_{1}}_{\M}+\norm{a_{2}}_{\M}+\norm q_{\M}<\norm{\rbr{\Re A}^{-1}}_{\M}^{-1}.\label{assu: Coercivity}
\end{equation}
This last requirement is  not as restrictive as it might seem to
be. It will turn out that the hypothesis~(\ref{assu: Coercivity})
is, in a sense, a weaker property than the relative $-\Delta$\nobreakdash-form-boundedness
of $a_{1}^{*}D+D^{*}a_{2}+q$ with relative bound zero, so that (\ref{assu: Coercivity})
is satisfied in most cases~-- see Remark~\ref{rem: Coercivity}
below.
\begin{rem}
Our intention is to replace the scale of Lebesgue spaces with that
of multiplier spaces, which prove to be perfectly suited to the problem
in question. In~particular, for this reason we use $\M\rbr{L_{2}\rbr{\Omega}}$
for the space~$L_{\infty}\rbr{\Omega}$.
\end{rem}
Let $iD_{1}$ be the vector of first partial derivatives with respect
to $x_{1}$ and $iD_{2}$, with respect to $x_{2}$. Let $\op D_{1}=\bigrbr{\genfrac{}{}{0pt}{}{D_{1}}{0}}$
and~$\op D_{2}=\bigrbr{\genfrac{}{}{0pt}{}{0}{D_{2}}}$.  We shall
 use $\op D$ to denote~$\op D_{1}+\op D_{2}$. Given $\varepsilon\in\set E=(0,1]$,
we introduce the notation~$\gamma^{\varepsilon}=\rbr{\op S^{1/\varepsilon}\otimes\op I}\gamma\rbr{\op S^{\varepsilon}\otimes\op I}$
for any Sobolev multiplier~$\gamma$ (if $\gamma$ is a function,
then $\gamma^{\varepsilon}\rbr x=\gamma\rbr{\varepsilon^{-1}x_{1},x_{2}}$
for $x\in\Xi$) and define the  form~$\form a^{\varepsilon}$ on
$H^{1}\rbr{\Xi}$ by
\begin{equation}
\form a^{\varepsilon}\sbr u=\rbr{A^{\varepsilon}\op Du,\op Du}_{2,\Xi}+\rbr{\op Du,a_{1}^{\varepsilon}u}_{2,\Xi}+\rbr{a_{2}^{\varepsilon}u,\op Du}_{2,\Xi}+\rbr{q^{\varepsilon}u,u}_{2,\Xi}.\label{def: a=001D4B}
\end{equation}
Notice that $\gamma\mapsto\gamma^{\varepsilon}$ is a bounded map
of $\M\rbr{H^{m}\rbr{\Omega},L_{2}\rbr{\Omega}}$ onto $\M\rbr{H^{m}\rbr{\Omega^{\varepsilon}},L_{2}\rbr{\Omega^{\varepsilon}}}$
and of $\M\rbr{H^{m}\rbr{\Omega},H^{1}\rbr{\Omega}^{*}}$ onto $\M\rbr{H^{m}\rbr{\Omega^{\varepsilon}},H^{1}\rbr{\Omega^{\varepsilon}}^{*}}$,
with norms not exceeding~$1$. Here $\Omega^{\varepsilon}=\varepsilon\Omega_{1}\times\Omega_{2}$.
Then, since $A^{\varepsilon}$, $a_{n}^{\varepsilon}$, $n\in\Natto 2$,
and~$q^{\varepsilon}$ are periodic with respect to $\varepsilon\Lambda$
and since $\cbr{\rbr{\Omega_{\lambda}}^{\varepsilon}}_{\lambda\in\Lambda}$
is a tiling of $\Xi$, we see that $A^{\varepsilon}\in\M\rbr{L_{2}\rbr{\Xi}}^{d\times d}$,
$a_{n}^{\varepsilon}\in\M\rbr{H^{1}\rbr{\Xi},L_{2}\rbr{\Xi}}^{d}$
and~$q^{\varepsilon}\in\M\rbr{H^{1}\rbr{\Xi},H^{1}\rbr{\Xi}^{*}}$.
Furthermore, the corresponding norms are majorized by the multiplier
norms of $A$, $a_{n}$ and~$q$, respectively. Now it is clear
that $\form a^{\varepsilon}$ is bounded,
\begin{equation}
\abs{\form a^{\varepsilon}\sbr{u,v}}\le C_{\flat}\norm u_{1,2,\Xi}\norm v_{1,2,\Xi},\qquad u,v\in H^{1}\rbr{\Xi},\label{est: a=001D4B is bounded}
\end{equation}
where
\[
C_{\flat}=\norm A_{\M}+\norm{a_{1}}_{\M}+\norm{a_{2}}_{\M}+\norm q_{\M}.
\]
Observe also that
\begin{equation}
\Re\form a^{\varepsilon}\sbr u\ge c_{*}\norm{\op Du}_{2,\Xi}^{2}-c_{\natural}\norm u_{2,\Xi}^{2},\qquad u\in H^{1}\rbr{\Xi},\label{est: a=001D4B is coercive}
\end{equation}
where
\begin{align}
c_{*} & =\norm{\rbr{\Re A}^{-1}}_{\M}^{-1}-\norm{a_{1}}_{\M}-\norm{a_{2}}_{\M}-\norm q_{\M},\label{def: c_*}\\
c_{\natural} & =2^{-1}\bigrbr{\norm{a_{1}}_{\M}+\norm{a_{2}}_{\M}}+\norm q_{\M}.\label{def: c_hash}
\end{align}
Since $c_{*}$ is positive, it follows that $\form a^{\varepsilon}$
is coercive.

Thus, $\form a^{\varepsilon}$ is strictly $m$\nobreakdash-sectorial,
with sector
\[
\set S_{1}=\bigcbr{z\in\C\colon\abs{\Im z}\le c_{*}^{-1}C_{\flat}\rbr{\Re z+c_{*}+c_{\natural}}}.
\]
Let $\op A_{\mu}^{\varepsilon}\colon H^{1}\rbr{\Xi}\to H^{1}\rbr{\Xi}^{*}$
be the operator associated with the form~$\form a_{\mu}^{\varepsilon}=\form a^{\varepsilon}-\mu$.
Then $\op A_{\mu}^{\varepsilon}$ is an isomorphism whenever~$\mu\notin\set S_{1}$.
\begin{rem}
\label{rem: Coercivity}The~hypothesis~(\ref{assu: Coercivity})
is needed in order for the form~$\form a^{\varepsilon}$ to be coercive.
In~fact, it can be weakened to allow those $a_{n}$, $n\in\Natto 2$,
and~$q$ that satisfy, for any~$u\in H^{1}\rbr{\Omega}$,
\begin{align}
\norm{a_{n}u}_{2,\Omega}^{2} & \le c_{a_{n}}\norm{Du}_{2,\Omega}^{2}+C_{a_{n}}\norm u_{2,\Omega}^{2},\label{est: a=002099 weakened}\\
\bigabs{\rbr{qu,u}_{2,\Omega}} & \le c_{q}\norm{Du}_{2,\Omega}^{2}+C_{q}\norm u_{2,\Omega}^{2}\label{est: q weakened}
\end{align}
with
\begin{equation}
c_{a_{1}}^{1/2}+c_{a_{2}}^{1/2}+c_{q}<\norm{\rbr{\Re A}^{-1}}_{\M}^{-1}.\label{est: a=002081, a=002082 and q weakened}
\end{equation}
Indeed, since we are interested in estimating operator norms (see
Theorems~\ref{thm: Convergence}\hairspace--\hairspace\ref{thm: Approximation with 2d corrector})
and since $\op S^{\delta}$ is an isomorphism, we may replace $\op A^{\varepsilon}$
by $\hat{\op A}^{\varepsilon}=\op S^{\delta}\op A^{\varepsilon}\rbr{\op S^{\delta}}^{-1}$.
(Here, we realize the torus~$\Omega_{2}=\T_{L}^{d_{2}}$ as the
cube~$\sbr{0,L}^{d_{2}}$ with opposite sides identified, and, in
this sense, $\delta^{-1}\T_{L}^{d_{2}}$ is well defined and equals~$\T_{\delta^{-1}L}^{d_{2}}$.)
 It~is~easy to see that the coefficients of $\hat{\op A}^{1}$
are given by $\hat{A}=\delta^{-2}\op S^{\delta}A\op S^{1/\delta}$,
$\hat{a}_{n}=\delta^{-1}\op S^{\delta}a_{n}\op S^{1/\delta}$ and~$\hat{q}=\op S^{\delta}q\op S^{1/\delta}$.
Therefore, if we take $\delta$ so that $\delta^{2}\le\min\bigcbr{c_{a_{1}}C_{a_{1}}^{-1},c_{a_{2}}C_{a_{2}}^{-1},c_{q}C_{q}^{-1}}$,
then
\[
\begin{aligned}\norm{\hat{a}_{1}}_{\M}+\norm{\hat{a}_{2}}_{\M}+\norm{\hat{q}}_{\M} & \le\delta^{-2}\bigrbr{c_{a_{1}}^{1/2}+c_{a_{2}}^{1/2}+c_{q}}\\
 & <\delta^{-2}\norm{\rbr{\Re A}^{-1}}_{\M}^{-1}=\norm{\rbr{\Re\hat{A}}^{-1}}_{\M}^{-1};
\end{aligned}
\]
that is, the hypothesis~(\ref{assu: Coercivity}) holds for~$\hat{\op A}^{\varepsilon}$.
We note that the class of operators such as $\op A^{\varepsilon}$
here is broad enough to  cover most cases that arise in applications~--
see an example below.
\end{rem}
We are interested in  approximations for $\rbr{\op A_{\mu}^{\varepsilon}}^{-1}$
and~$\op D\rbr{\op A_{\mu}^{\varepsilon}}^{-1}$ in the operator
norm on~$L_{2}\rbr{\Xi}$. In~order to describe these approximations,
we  define the effective operator and two different correctors.

\subsection{Effective operator}

Let $N$ be the weak solution of
\begin{equation}
\op D_{1}^{*}A\rbr{\op D_{1}N+I}=0\label{def: N}
\end{equation}
in~$L_{2}\rbr{\Omega_{2};\widetilde H_{0}^{1}\rbr{\Omega_{1}}}^{1\times d}$,
and $M$ be the weak solution of
\begin{equation}
\op D_{1}^{*}\rbr{A\op D_{1}M+a_{2}}=0\label{def: M}
\end{equation}
in~$L_{2}\rbr{\Omega_{2};\widetilde H_{0}^{1}\rbr{\Omega_{1}}}$. We know
that $N$ and~$M$ exist and are unique, since we may rewrite these
problems as
\begin{equation}
D_{1}^{*}A_{11}D_{1}u=D_{1}^{*}f,\label{eq: Form of the auxiliary problems}
\end{equation}
with an $f$ in $L_{2}\rbr{\Omega}^{d_{1}}$ and $u\in L_{2}\rbr{\Omega_{2};\widetilde H_{0}^{1}\rbr{\Omega_{1}}}$
to be found. Notice in passing that such a $u$  satisfies
\begin{equation}
\bigrbr{A_{11}\rbr{\wc,x_{2}}\,D_{1}u\rbr{\wc,x_{2}}-f\rbr{\wc,x_{2}},D_{1}v}_{2,\Omega_{1}}=0\label{eq: The auxiliary problems on =0003A9=002081}
\end{equation}
for~almost every~$x_{2}\in\Omega_{2}$ and~all~$v\in\widetilde H^{1}\rbr{\Omega_{1}}$.

We now provide some elementary properties of $N$ and~$M$ (cf.~\cite[Proposition~8.2]{Su2}).
\begin{lem}
\label{lem: The auxiliary problem solution as multiplicators}Let
$u$ be the weak solution of the problem~\textup{(\ref{eq: Form of the auxiliary problems})}
where the function~$f$ is in $\M\rbr{H^{m}\rbr{\Omega};L_{2}\rbr{\Omega}}^{d_{1}}$,
$m$ a non-negative integer. Then $D_{1}u\in\M\rbr{H^{m}\rbr{\Omega_{2}},L_{2}\rbr{\Omega}}^{d_{1}}$
and
\[
\norm{D_{1}u}_{\M}\le\abs{\Omega_{1}}^{1/2}\norm{\rbr{\Re A}^{-1}}_{\M}\norm f_{\M}.
\]
If, in addition, $D_{2,n}f\in\M\rbr{H^{m}\rbr{\Omega};L_{2}\rbr{\Omega}}^{d_{1}}$
for some $n\in\Natto{d_{2}}$, then $D_{2,n}D_{1}u\in\M\rbr{H^{m}\rbr{\Omega_{2}},L_{2}\rbr{\Omega}}^{d_{1}}$
and
\[
\norm{D_{2,n}D_{1}u}_{\M}\le\abs{\Omega_{1}}^{1/2}\norm{\rbr{\Re A}^{-1}}_{\M}\bigrbr{\norm{D_{2,n}A}_{\M}\norm{\rbr{\Re A}^{-1}}_{\M}\norm f_{\M}+\norm{D_{2,n}f}_{\M}}.
\]
\end{lem}
\begin{proof}
Let $v=u\abs w^{2}$ with~$w\in C^{m}\rbr{\Omega_{2}}$. Then $v\in L_{2}\rbr{\Omega_{2};\widetilde H_{0}^{1}\rbr{\Omega_{1}}}$,
and we can apply both sides of (\ref{eq: Form of the auxiliary problems})
to $v$, obtaining
\[
\norm{\rbr{D_{1}u}w}_{2,\Omega}\le\norm{\rbr{\Re A_{11}}^{-1}}_{\M}\norm{fw}_{2,\Omega}.
\]
This proves the first assertion.

Suppose now  $D_{2,n}f\in\M\rbr{H^{m}\rbr{\Omega};L_{2}\rbr{\Omega}}^{d_{1}}$.
We know that $D_{2,n}f\in L_{2}\rbr{\Omega}^{d_{1}}$, so $D_{2,n}u$
exists and belongs to $L_{2}\rbr{\Omega_{2};\widetilde H_{0}^{1}\rbr{\Omega_{1}}}$,
which may be verified by using the difference quotient technique
of \textsc{Nirenberg}. Therefore, we can write
\[
D_{1}^{*}A_{11}D_{1}D_{2,n}u=D_{1}^{*}\bigrbr{D_{2,n}f-\rbr{D_{2,n}A_{11}}D_{1}u}.
\]
Applying both sides of the last equality to $v=\rbr{D_{2,n}u}\abs w^{2}$
with $w\in C^{m}\rbr{\Omega_{2}}$ yields
\[
\norm{\rbr{D_{1}D_{2,n}u}w}_{2,\Omega}\le\norm{\rbr{\Re A_{11}}^{-1}}_{\M}\bigrbr{\norm{\rbr{D_{2,n}A_{11}}\rbr{D_{1}u}w}_{2,\Omega}+\norm{\rbr{D_{2,n}f}w}_{2,\Omega}},
\]
and~the second assertion follows.\qedspace
\end{proof}
From~the~above lemma and the Poincar\'{e} inequality~(\ref{est: Poincar=0000E9's inequality}),
we conclude that $N\in\S\rbr{L_{2}\rbr{\Omega_{2}},L_{2}\rbr{\Omega}}^{1\times d}$
and $\op D_{1}N\in\S\rbr{L_{2}\rbr{\Omega_{2}},L_{2}\rbr{\Omega}}^{d\times d}$,
while $M\in\S\rbr{H^{1}\rbr{\Omega_{2}},L_{2}\rbr{\Omega}}$ and
$\op D_{1}M\in\S\rbr{H^{1}\rbr{\Omega_{2}},L_{2}\rbr{\Omega}}^{d}$.

We now turn to the effective coefficients.

Let
\begin{equation}
A^{0}=\abs{\Omega_{1}}^{-1}\int_{\Omega_{1}}A\rbr{\op D_{1}N+I}\,dy_{1}.\label{def: A=002070}
\end{equation}
Then, from the properties of $A$ and~$N$, we have $A^{0}\in\S\rbr{L_{2}\rbr{\Omega_{2}}}^{d\times d}$.
It~is a standard fact (see~\cite[Section~1.6]{ZhKO}) that if $\Re A$
is positive definite, then
\begin{equation}
\Re A^{0}\ge\biggrbr{\abs{\Omega_{1}}^{-1}\int_{\Omega_{1}}\rbr{\Re A}^{-1}dy_{1}}^{-1}.\label{est: Lower bound for ReA=002070}
\end{equation}
This implies that $\Re A^{0}$ is also positive definite and furthermore
$\rbr{\Re A^{0}}^{-1}$ is in $\M\rbr{L_{2}\rbr{\Omega_{2}}}^{d\times d}$
and $\norm{\rbr{\Re A^{0}}^{-1}}_{\M}\le\norm{\rbr{\Re A}^{-1}}_{\M}.$

Next, we define the functions
\begin{align}
a_{1}^{0} & =\abs{\Omega_{1}}^{-1}\int_{\Omega_{1}}\rbr{\op D_{1}N+I}^{*}a_{1}\,dy_{1},\label{def: (a=002081)=002070}\\
a_{2}^{0} & =\abs{\Omega_{1}}^{-1}\int_{\Omega_{1}}\rbr{A\op D_{1}M+a_{2}}\,dy_{1}.\label{def: (a=002082)=002070}
\end{align}
Both of these are in $\S\rbr{H^{1}\rbr{\Omega_{2}},L_{2}\rbr{\Omega_{2}}}^{d}$,
as can be seen from the properties of $A$, $a_{1}$, $a_{2}$ and~$N$,
$M$.

Finally, let $q^{0}$ correspond to the form 
\begin{equation}
\rbr{q^{0}u,u}_{2,\Omega_{2}}=\abs{\Omega_{1}}^{-1}\rbr{qu,u}_{2,\Omega}+\abs{\Omega_{1}}^{-1}\rbr{a_{1}^{*}\op D_{1}Mu,u}_{2,\Omega}\label{def: q=002070}
\end{equation}
on~$H^{1}\rbr{\Omega_{2}}$. By~the~properties of $a_{1}$, $q$
and~$M$, we obtain~$q^{0}\in\S\rbr{H^{1}\rbr{\Omega_{2}},H^{1}\rbr{\Omega_{2}}^{*}}$.
Notice that, in the case when $q$ is a function, we have, as usual,
\[
q^{0}=\abs{\Omega_{1}}^{-1}\int_{\Omega_{1}}q\,dy_{1}+\abs{\Omega_{1}}^{-1}\int_{\Omega_{1}}a_{1}^{*}\op D_{1}M\,dy_{1}.
\]

We are almost ready to define the effective operator. Consider the
form~$\form a^{0}$ on $H^{1}\rbr{\Xi}$ given by
\begin{equation}
\form a^{0}\sbr u=\rbr{A^{0}\op Du,\op Du}_{2,\Xi}+\rbr{\op Du,a_{1}^{0}u}_{2,\Xi}+\rbr{a_{2}^{0}u,\op Du}_{2,\Xi}+\rbr{q^{0}u,u}_{2,\Xi}\label{def: a=002070}
\end{equation}
where
\[
\rbr{q^{0}u,u}_{2,\Xi}=\int_{\R^{d_{1}}}\rbr{q^{0}u\rbr{x_{1},\wc},u\rbr{x_{1},\wc}}_{2,\Omega_{2}}dx_{1}.
\]
Then $\form a^{0}$ is plainly bounded,
\begin{equation}
\abs{\form a^{0}\sbr{u,v}}\le C_{\flat}^{0}\norm u_{1,2,\Xi}\norm v_{1,2,\Xi},\qquad u,v\in H^{1}\rbr{\Xi},\label{est: a=002070 is bounded}
\end{equation}
with
\[
C_{\flat}^{0}=\norm{A^{0}}_{\M}+\norm{a_{1}^{0}}_{\M}+\norm{a_{2}^{0}}_{\M}+\norm{q^{0}}_{\M}.
\]
In~a~moment, we shall see that it is coercive.
\begin{lem}
\label{lem: =0001CE=002070 is coercive}Let $\check{\form a}^{0}$
be the form on $H^{1}\rbr{\Xi}\oplus L_{2}\rbr{\Xi;\widetilde H_{0}^{1}\rbr{\Omega_{1}}}$
given by
\[
\begin{aligned}\check{\form a}^{0}\sbr{\check{u}} & =\smash[b]{\abs{\Omega_{1}}^{-1}\int_{\Xi}\int_{\Omega_{1}}\Bigl(\bigabr{A\rbr{y_{1},x_{2}}\,\check{\op D}\check{u}\rbr{x,y_{1}},\check{\op D}\check{u}\rbr{x,y_{1}}}}\\
 & \hphantom{{}=\abs{\Omega_{1}}^{-1}\int_{\Xi}\int_{\Omega_{1}}\Bigl(\bigabr{A\rbr{y_{1},x_{2}}\,\check{\op D}\check{u}\rbr{x,y_{1}},\check{\op D}\check{u}\rbr{x,y_{1}}}}\mathllap{{}+\bigabr{\check{\op D}\check{u}\rbr{x,y_{1}},a_{1}\rbr{y_{1},x_{2}}\,\check{u}_{1}\rbr x}}\\
 & \hphantom{{}=\abs{\Omega_{1}}^{-1}\int_{\Xi}\int_{\Omega_{1}}\Bigl(\bigabr{A\rbr{y_{1},x_{2}}\,\check{\op D}\check{u}\rbr{x,y_{1}},\check{\op D}\check{u}\rbr{x,y_{1}}}}\mathllap{{}+\bigabr{a_{2}\rbr{y_{1},x_{2}}\,\check{u}_{1}\rbr x,\check{\op D}\check{u}\rbr{x,y_{1}}}}\Bigr)\,dx\,dy_{1}\\
 & \quad+\smash[t]{\abs{\Omega_{1}}^{-1}\int_{\R^{d_{1}}}\rbr{q\check{u}_{1}\rbr{x_{1},\wc},\check{u}_{1}\rbr{x_{1},\wc}}_{2,\Omega}dx_{1}}
\end{aligned}
\]
where $\check{u}=\rbr{\check{u}_{1},\check{u}_{2}}$ and~$\check{\op D}\check{u}\rbr{x,y_{1}}=\op D_{x}\check{u}_{1}\rbr x+\op D_{y_{1}}\check{u}_{2}\rbr{x,y_{1}}$.
Then $\check{\form a}^{0}$ is  coercive and
\begin{equation}
\Re\check{\form a}^{0}\sbr{\check{u}}\ge c_{*}\abs{\Omega_{1}}^{-1}\norm{\check{\op D}\check{u}}_{2,\Xi\times\Omega_{1}}^{2}-c_{\natural}\norm{\check{u}_{1}}_{2,\Xi}^{2},\label{est: =0001CE=002070 is coercive}
\end{equation}
all~$\check{u}\in H^{1}\rbr{\Xi}\oplus L_{2}\rbr{\Xi;\widetilde H_{0}^{1}\rbr{\Omega_{1}}}$.\end{lem}
\begin{proof}
While the proof is quite similar to that of~(\ref{est: a=001D4B is coercive}),
there is a difference: the variables~$x_{1}$ and~$y_{1}$ in the
definition of $\check{\form a}^{0}$ are ``mixed'', so that we
cannot treat the lower-order terms as before.

We begin with a first-order term. By~Cauchy's inequality, we have
\[
\begin{aligned}\hspace{2em} & \hspace{-2em}\biggabs{\int_{\Xi}\int_{\Omega_{1}}\bigabr{\check{\op D}\check{u}\rbr{x,y_{1}},a_{1}\rbr{y_{1},x_{2}}\,\check{u}_{1}\rbr x}\,dx\,dy_{1}}\\
 & \le\norm{\check{\op D}\check{u}}_{2,\Xi\times\Omega_{1}}\smash[t]{\biggrbr{\int_{\Xi}\int_{\Omega_{1}}\abs{a_{1}\rbr{y_{1},x_{2}}\,\check{u}_{1}\rbr x}^{2}dx\,dy_{1}}^{1/2}}.
\end{aligned}
\]
Let $v$ denote the mapping~$y\mapsto\bigrbr{\int_{\R^{d_{1}}}\abs{\check{u}_{1}\rbr{x_{1},y_{2}}}^{2}dx_{1}}^{1/2}$.
Then $v\in H^{1}\rbr{\Omega}$, with $\norm v_{1,2,\Omega}\le\abs{\Omega_{1}}^{1/2}\norm{\check{u}_{1}}_{1,2,\Xi}$,
and
\[
\int_{\Xi}\int_{\Omega_{1}}\abs{a_{1}\rbr{y_{1},x_{2}}\,\check{u}_{1}\rbr x}^{2}dx\,dy_{1}=\norm{a_{1}v}_{2,\Omega}^{2}.
\]
As~a~result,
\[
\begin{aligned}\hspace{2em} & \hspace{-2em}\biggabs{\int_{\Xi}\int_{\Omega_{1}}\bigabr{\check{\op D}\check{u}\rbr{x,y_{1}},a_{1}\rbr{y_{1},x_{2}}\,\check{u}_{1}\rbr x}\,dx\,dy_{1}}\\
 & \le\abs{\Omega_{1}}^{1/2}\norm{a_{1}}_{\M}\norm{\check{\op D}\check{u}}_{2,\Xi\times\Omega_{1}}\norm{\check{u}_{1}}_{1,2,\Xi}\\
 & \le\norm{a_{1}}_{\M}\bigrbr{\norm{\check{\op D}\check{u}}_{2,\Xi\times\Omega_{1}}^{2}+2^{-1}\abs{\Omega_{1}}\norm{\check{u}_{1}}_{2,\Xi}^{2}}.
\end{aligned}
\]
We have used here the fact that, by Stokes' theorem,
\[
\norm{\check{\op D}\check{u}}_{2,\Xi\times\Omega_{1}}^{2}=\abs{\Omega_{1}}\norm{\op D_{x}\check{u}_{1}}_{2,\Xi}^{2}+\norm{\op D_{y_{1}}\check{u}_{2}}_{2,\Xi\times\Omega_{1}}^{2}.
\]

We may likewise prove that
\[
\begin{aligned}\hspace{2em} & \hspace{-2em}\biggabs{\int_{\Xi}\int_{\Omega_{1}}\bigabr{a_{2}\rbr{y_{1},x_{2}}\,\check{u}_{1}\rbr x,\check{\op D}\check{u}\rbr{x,y_{1}}}\,dx\,dy_{1}}\\
 & \le\norm{a_{2}}_{\M}\bigrbr{\norm{\check{\op D}\check{u}}_{2,\Xi\times\Omega_{1}}^{2}+2^{-1}\abs{\Omega_{1}}\norm{\check{u}_{1}}_{2,\Xi}^{2}}
\end{aligned}
\]
and
\[
\begin{aligned}\hspace{2em} & \hspace{-2em}\biggabs{\int_{\R^{d_{1}}}\rbr{q\check{u}_{1}\rbr{x_{1},\wc},\check{u}_{1}\rbr{x_{1},\wc}}_{2,\Omega}dx_{1}}\\
 & =\bigabs{\rbr{qv,v}_{2,\Omega}}\le\norm q_{\M}\bigrbr{\norm{\check{\op D}\check{u}}_{2,\Xi\times\Omega_{1}}^{2}+\abs{\Omega_{1}}\norm{\check{u}_{1}}_{2,\Xi}^{2}}.
\end{aligned}
\]
Combining these inequalities with
\[
\begin{aligned}\hspace{2em} & \hspace{-2em}\Re\int_{\Xi}\int_{\Omega_{1}}\bigabr{A\rbr{y_{1},x_{2}}\,\check{\op D}\check{u}\rbr{x,y_{1}},\check{\op D}\check{u}\rbr{x,y_{1}}}\,dx\,dy_{1}\\
 & \ge\norm{\rbr{\Re A}^{-1}}_{\M}^{-1}\norm{\check{\op D}\check{u}}_{2,\Xi\times\Omega_{1}}^{2},
\end{aligned}
\]
which is obvious, gives~(\ref{est: =0001CE=002070 is coercive}).\qedspace\end{proof}
\begin{rem}
The~form~$\check{\form a}^{0}$ is associated with the two-scale
homogenized system, first proposed by \textsc{Allaire}~\cite{Al}
in the context of two-scale convergence. See also \cite{LNW} for
a self-contained approach to this matter.
\end{rem}
Now we wish to relate the form~$\check{\form a}^{0}$ to~$\form a^{0}$.
Fix~a~$u\in H^{1}\rbr{\Xi}$. Let $\check{u}_{1}\rbr x=u\rbr x$
and $\check{u}_{2}\rbr{x,y_{1}}=N\rbr{y_{1},x_{2}}\,\op D_{x}u\rbr x+M\rbr{y_{1},x_{2}}\,u\rbr x$.
We claim that $\check{u}$ belongs to~$H^{1}\rbr{\Xi}\oplus L_{2}\rbr{\Xi;\widetilde H_{0}^{1}\rbr{\Omega_{1}}}$.
Indeed, $\check{u}_{2}\rbr{x,y_{1}}$  has a derivative with respect
to $y_{1}$, and, by reasoning explained in the proof of Lemma~\ref{lem: =0001CE=002070 is coercive},
 it lies in $L_{2}\rbr{\Xi\times\Omega_{1}}^{d_{1}}$ (notice here
that $\op D_{1}N$ and~$\op D_{1}M$ are multipliers). Applying identities
for $N$ and~$M$ in the form~(\ref{eq: The auxiliary problems on =0003A9=002081}),
we find that $\check{\form a}^{0}\sbr{\check{u},\check{v}}=0$ for
all~$\check{v}\in\cbr 0\oplus L_{2}\rbr{\Xi;\widetilde H_{0}^{1}\rbr{\Omega_{1}}}$,
and therefore $\check{\form a}^{0}\sbr{\check{u}}=\check{\form a}^{0}\sbr{\check{u},u\oplus0}$.
Now it follows from the definitions of the effective coefficients
that
\[
\check{\form a}^{0}\sbr{\check{u}}=\form a^{0}\sbr u
\]
for~every~$u\in H^{1}\rbr{\Xi}$, which is the desired relation.

Since $\check{\form a}^{0}$ is coercive, the above identity tells
us that so is $\form a^{0}$, with
\begin{equation}
\Re\form a^{0}\sbr u\ge c_{*}\norm{\op Du}_{2,\Xi}^{2}-c_{\natural}\norm u_{2,\Xi}^{2},\qquad u\in H^{1}\rbr{\Xi}.\label{est: a=002070 is coercive}
\end{equation}
Hence, the form~$\form a^{0}$ is strictly $m$\nobreakdash-sectorial,
with sector
\[
\set S_{0}=\bigcbr{z\in\C\colon\abs{\Im z}\le c_{*}^{-1}C_{\flat}^{0}\rbr{\Re z+c_{*}+c_{\natural}}}.
\]
Corresponding to $\form a_{\mu}^{0}=\form a^{0}-\mu$ there is an
 operator~$\op A_{\mu}^{0}=\op A^{0}-\mu\colon H^{1}\rbr{\Xi}\to H^{1}\rbr{\Xi}^{*}$,
which is an isomorphism provided that~$\mu\notin\set S_{0}$.  For~such
a $\mu$, $\rbr{\op A_{\mu}^{0}}^{-1}$ maps $L_{2}\rbr{\Xi}$ onto~$H^{2}\rbr{\Xi}$.
(This can be shown by using the difference quotient technique of \textsc{Nirenberg};
see the proof of Lemma~\ref{lem: Estimates for A=002070(=0003C4)}
for further details on this matter.) We denote the largest of the
sectors~$\set S_{0}$ and~$\set S_{1}$ by~$\set S$.

\subsection{Correctors}

We introduce two types of correctors. The~first, denoted $\op K_{\mu}^{\varepsilon}$,
will be needed  to obtain the approximation  for $\op D_{1}\rbr{\op A_{\mu}^{\varepsilon}}^{-1}$
and is defined as follows. Let $\op P^{\varepsilon}$ be the pseudodifferential
operator in the $x_{1}$\nobreakdash-variable with symbol~$\chi_{\varepsilon^{-1}\Omega_{1}^{*}}$,
where $\chi_{\varepsilon^{-1}\Omega_{1}^{*}}$ is the characteristic
function of the set~$\varepsilon^{-1}\Omega_{1}^{*}$, or, to put
it differently,
\[
\op P^{\varepsilon}=\rbr{\op F\otimes\op I}^{*}\chi_{\varepsilon^{-1}\Omega_{1}^{*}}\rbr{\op F\otimes\op I}.
\]
Here $\op F$ is the Fourier transform in~$L_{2}\rbr{\R^{d_{1}}}$.
Then the corrector~$\op K_{\mu}^{\varepsilon}\colon L_{2}\rbr{\Xi}\to H^{1}\rbr{\Xi}$
for $\op A^{\varepsilon}$ is given by
\begin{equation}
\op K_{\mu}^{\varepsilon}=\rbr{N^{\varepsilon}\op D+M^{\varepsilon}}\rbr{\op A_{\mu}^{0}}^{-1}\op P^{\varepsilon}.\label{def: K=001D4B}
\end{equation}
We remark that,  while $\rbr{N^{\varepsilon}\op D+M^{\varepsilon}}\rbr{\op A_{\mu}^{0}}^{-1}f$,
with $f\in L_{2}\rbr{\Xi}$, is not generally in $H^{1}\rbr{\Xi}$
(not even in~$L_{2}\rbr{\Xi}$), the function~$\op K_{\mu}^{\varepsilon}f$
always is, which may be proved by applying the scaling transformation
and the Gelfand transform (see~(\ref{eq: Direct integral for K=001D4B}))
and then using the properties of $N$ and~$M$ (see Lemma~\ref{lem: Estimates for K(=0003C4)}).
What is more, these calculations  show that $\op K_{\mu}^{\varepsilon}$
is a bounded operator.

The~second corrector, denoted $\op C_{\mu}^{\varepsilon}$, will
be needed for a more subtle result. If~$k$ is a vector in~$\R^{d_{1}}$
and $\vect k$ is the corresponding element of~$\R^{d_{1}}\oplus\cbr 0$,
then we define differential expressions
\begin{align*}
\op S\rbr{k;y_{1}} & =\bigrbr{\rbr{\vect k+\op D_{2}}^{*}A\rbr{y_{1},\wc}+a_{1}^{*}\rbr{y_{1},\wc}}\rbr{\vect k+\op D_{2}}+\rbr{\vect k+\op D_{2}}^{*}a_{2}\rbr{y_{1},\wc}+q\rbr{y_{1},\wc},\\
\op T\rbr{k;y_{1}} & =\bigrbr{\rbr{\vect k+\op D_{2}}^{*}A\rbr{y_{1},\wc}+a_{1}^{*}\rbr{y_{1},\wc}}\op D_{y_{1}}
\end{align*}
and~families of operators
\begin{align*}
\op A_{\mu}^{0}\rbr k & =\rbr{\vect k+\op D_{2}}^{*}A^{0}\rbr{\vect k+\op D_{2}}+\rbr{a_{1}^{0}}^{*}\rbr{\vect k+\op D_{2}}+\rbr{\vect k+\op D_{2}}^{*}a_{2}^{0}+q^{0}-\mu,\\
\op K_{\mu}\rbr{k;y_{1}} & =\bigrbr{N\rbr{y_{1},\wc}\,\rbr{\vect k+\op D_{2}}+M\rbr{y_{1},\wc}}\rbr{\op A_{\mu}^{0}\rbr k}^{-1}.
\end{align*}
Let $\rbr{\op A_{\mu}^{\varepsilon}}^{+}$ be the adjoint of~$\op A_{\mu}^{\varepsilon}$.
For~the operator~$\rbr{\op A_{\mu}^{\varepsilon}}^{+}$, we construct
the effective operator~$\rbr{\op A_{\mu}^{0}}^{+}$ and the corrector~$\rbr{\op K_{\mu}^{\varepsilon}}^{+}$,
as well as the families~$\op A_{\mu}^{0}\rbr k^{+}$ and~$\op K_{\mu}\rbr{k;y_{1}}^{+}$.
(It may be noted in passing that $\rbr{\op A_{\mu}^{0}}^{+}$ is
the adjoint of~$\op A_{\mu}^{0}$.) Finally, let $\op L_{\mu}$ be
the pseudodifferential operator in the $x_{1}$\nobreakdash-variable
with operator-valued symbol~$k\mapsto\op L_{\mu}\rbr k\colon L_{2}\rbr{\Omega_{2}}\to L_{2}\rbr{\Omega_{2}}$
where 
\[
\op L_{\mu}\rbr k=\abs{\Omega_{1}}^{-1}\int_{\Omega_{1}}\bigrbr{\op K_{\mu}\rbr{k;y_{1}}^{+}}^{*}\bigrbr{\op S\rbr{k;y_{1}}\,\rbr{\op A_{\mu}^{0}\rbr k}^{-1}+\op T\rbr{k;y_{1}}\,\op K_{\mu}\rbr{k;y_{1}}}\,dy_{1};
\]
that is,
\[
\op L_{\mu}=\rbr{\op F\otimes\op I}^{*}\op L_{\mu}\rbr{\wc}\rbr{\op F\otimes\op I}.
\]
The~operator~$\op L_{\mu}^{+}$ is constructed similarly. The~corrector~$\op C_{\mu}^{\varepsilon}\colon L_{2}\rbr{\Xi}\to L_{2}\rbr{\Xi}$
is then defined by the formula
\begin{equation}
\op C_{\mu}^{\varepsilon}=\bigrbr{\op K_{\mu}^{\varepsilon}-\op L_{\mu}}+\bigrbr{\rbr{\op K_{\mu}^{\varepsilon}}^{+}-\op L_{\mu}^{+}}^{*}.\label{def: C=001D4B}
\end{equation}
We will see in what follows that $\op C_{\mu}^{\varepsilon}$ is
continuous.
\begin{rem}
Notice that, since $\op A_{\mu}^{0}\rbr{\wc}$ is the symbol of $\op A_{\mu}^{0}$
in the above indicated sense, $\op L_{\mu}$ can be written~as
\[
\op L_{\mu}=\rbr{\op A_{\mu}^{0}}^{-1}\op M\rbr{\op A_{\mu}^{0}}^{-1},
\]
where~$\op M\colon H^{2}\rbr{\Xi}\to H^{1}\rbr{\Xi}^{*}$ is a third-order
differential operator with coefficients depending only on~$x_{2}$.
\end{rem}

We conclude this section with an example of the operator~$\op A^{\varepsilon}$.

\subsection{An example}

Let $d>1$ and~$p>d$. From the~Ehrling lemma,  we know that if
$\gamma\in L_{p}\rbr{\Omega}$, then $\gamma\in\M\rbr{H^{1}\rbr{\Omega},L_{2}\rbr{\Omega}}$
and for all $\epsilon>0$ there is a $C_{\gamma}\rbr{\epsilon}>0$,
depending  on $d$, $p$, $\Omega$ and~$\norm{\gamma}_{p,\Omega}$,
such that 
\begin{equation}
\norm{\gamma u}_{2,\Omega}^{2}\le\epsilon\norm{Du}_{2,\Omega}^{2}+C_{\gamma}\rbr{\epsilon}\norm u_{2,\Omega}^{2},\qquad u\in H^{1}\rbr{\Omega}.\label{est: Gag-Nir on =0003A9}
\end{equation}
As~an~example of a multiplier between $H^{1}\rbr{\Omega}$ and~$H^{1}\rbr{\Omega}^{*}$,
 let $\delta_{\Sigma}$ be the Dirac distribution on a $d-1$ dimensional
Lipschitz surface~$\Sigma$ in $\Omega$ and let $\sigma$ be a function
in~$L_{p-1}\rbr{\Sigma}$. Again, for each $\epsilon>0$ there
is a $C_{\sigma}\rbr{\epsilon}$, depending  on $d$, $p$, $\Omega$,
$\Sigma$ and~$\norm{\sigma}_{p-1,\Sigma}$, such that 
\begin{equation}
\bigabs{\rbr{\sigma\delta_{\Sigma}u,u}_{2,\Omega}}\le\epsilon\norm{Du}_{2,\Omega}^{2}+C_{\sigma}\rbr{\epsilon}\norm u_{2,\Omega}^{2},\qquad u\in H^{1}\rbr{\Omega}.\label{est: Gag-Nir on =0003A3}
\end{equation}

Equipped with this information, we consider a periodic operator on
$L_{2}\rbr{\Xi}$ of the form
\[
\op H^{\varepsilon}=\rbr{\op D-A_{1}^{\varepsilon}}^{*}g^{\varepsilon}\rbr{\op D-A_{2}^{\varepsilon}}+V^{\varepsilon}.
\]
We may think of $\op H^{\varepsilon}$ as a (possibly non-self-adjoint)
periodic Schr\"{o}dinger operator with magnetic and electric potentials
that is associated with metric~$g^{\varepsilon}$. Suppose that $g$
is a periodic function in $\Lip\rbr{\Omega_{2};L_{\infty}\rbr{\Omega_{1}}}^{d\times d}$
and $\Re g$ is uniformly positive definite. Let $A_{1}$ and~$A_{2}$
be periodic functions in $W_{p}^{1}\rbr{\Omega_{2};L_{p}\rbr{\Omega_{1}}}^{d}$.
Finally, let $\Sigma$ be a $d-1$ dimensional periodic Lipschitz
surface in~$\Xi$. Then we assume that $V$ is the sum of a periodic
function~$\hat{V}\in W_{p/2}^{1}\rbr{\Omega_{2};L_{p/2}\rbr{\Omega_{1}}}$
and a distribution~$\sigma\delta_{\Sigma}$ with periodic~$\sigma\in W_{p-1}^{1}\rbr{\Sigma\cap\Omega}$.
Clearly, $\op H^{\varepsilon}$ thus defined can be expressed in the
form
\[
\op H^{\varepsilon}=\op D^{*}A^{\varepsilon}\op D+\rbr{a_{1}^{\varepsilon}}^{*}\op D+\op D^{*}a_{2}^{\varepsilon}+q^{\varepsilon},
\]
where the coefficients satisfy the properties~(\ref{est: a=002099 weakened})\hairspace--\hairspace(\ref{est: a=002081, a=002082 and q weakened})
in Remark~\ref{rem: Coercivity}. So our result applies to~$\op H^{\varepsilon}$.

It~is straightforward to construct an analogous example for the
case~$d=1$. Now we take $\Sigma$ to be a discrete periodic set
of points in $\R$ and assume that $g\in L_{\infty}\rbr{\Omega}$
with $\Re g$ uniformly positive definite, $A_{1},A_{2}\in L_{2}\rbr{\Omega}$
and $V=\hat{V}+\sigma\delta_{\Sigma}$ where $\hat{V}$ lies in $L_{1}\rbr{\Omega}$
and $\sigma$ is a periodic function on~$\Sigma$.

We note that the potential~$V^{\varepsilon}$ may also involve a
singular term~$\varepsilon^{-1}W^{\varepsilon}$ with a suitable
function~$W$. We refer the reader to~\cite[Section~11]{Su3} for
the details.

\section{\label{sec: Main results}Main results}

We now state the principal results of the present paper.
\begin{thm}
\label{thm: Convergence}If~$\mu\notin\set S$, then for any $\varepsilon\in\set E$
we have 
\begin{align}
\norm{\rbr{\op A_{\mu}^{\varepsilon}}^{-1}-\rbr{\op A_{\mu}^{0}}^{-1}}_{\B\rbr{L_{2}\rbr{\Xi}}} & \lesssim\varepsilon,\label{est: Convergence}\\
\norm{\op D_{2}\rbr{\op A_{\mu}^{\varepsilon}}^{-1}-\op D_{2}\rbr{\op A_{\mu}^{0}}^{-1}}_{\B\rbr{L_{2}\rbr{\Xi}}^{d}} & \lesssim\varepsilon.\label{est: Convergence of composition with D2}
\end{align}
The~estimates are sharp with respect to the order, and the constants
depend only on $r_{\Lambda}$, $\mu$ and the multiplier norms of
the coefficients.
\end{thm}

\begin{thm}
\label{thm: Approximation with 1st corrector}If~$\mu\notin\set S$,
then for any $\varepsilon\in\set E$ we have 
\begin{equation}
\norm{\op D_{1}\rbr{\op A_{\mu}^{\varepsilon}}^{-1}-\op D_{1}\rbr{\op A_{\mu}^{0}}^{-1}-\varepsilon\op D_{1}\op K_{\mu}^{\varepsilon}}_{\B\rbr{L_{2}\rbr{\Xi}}^{d}}\lesssim\varepsilon.\label{est: Approximation with 1st corrector}
\end{equation}
The~estimate is sharp with respect to the order, and the constant
depends only on $r_{\Lambda}$, $\mu$ and the multiplier norms of
the coefficients.
\end{thm}

\begin{thm}
\label{thm: Approximation with 2d corrector}If~$\mu\notin\set S$,
then for any $\varepsilon\in\set E$ we have
\begin{equation}
\norm{\rbr{\op A_{\mu}^{\varepsilon}}^{-1}-\rbr{\op A_{\mu}^{0}}^{-1}-\varepsilon\op C_{\mu}^{\varepsilon}}_{\B\rbr{L_{2}\rbr{\Xi}}}\lesssim\varepsilon^{2}.\label{est: Approximation with 2d corrector}
\end{equation}
The~estimate is sharp with respect to the order, and the constant
depends only on $r_{\Lambda}$, $\mu$ and the multiplier norms of
the coefficients.\end{thm}
\begin{rem}
Although it is possible to write down all the constants explicitly,
we do not do so here. In~particular, we write $\alpha\lesssim\beta$
 to mean $\alpha\le C\beta$  where $C$ is a positive constant
depending only on $r_{\Lambda}$, $\mu$ and the multiplier norms
of the coefficients.
\end{rem}

\begin{rem}
Using the resolvent identity, we can transfer the estimates in Theorems~\ref{thm: Convergence}\hairspace--\hairspace\ref{thm: Approximation with 2d corrector}
to those $\mu\in\set S$ for which $\op A_{\mu}^{\varepsilon}$ (at~least
for each $\varepsilon$ in an interval~$(0,\varepsilon_{\mu}]$)
and~$\op A_{\mu}^{0}$, when viewed as operators on~$L_{2}\rbr{\Xi}$,
have bounded inverses with norms majorized by constants independent
of~$\varepsilon$. For~instance, the estimates hold if $\mu$ is
in the resolvent set of the effective operator, but in this case we
have no control over~$\varepsilon_{\mu}$.
\end{rem}

\begin{rem}
The~hypothesis that the coefficients have weak derivatives with respect
to the non-periodic variable is crucial to our analysis and reflects
the fact that the roles of the two variables are quite different.
Roughly speaking, only the first variable is involved in the homogenization
procedure, while the second plays the role of a parameter (see,
for example, the definitions of $N$ and~$M$, where this is literally
the case). In~particular, the hypothesis ensures that $N$ and~$M$
belong to $H^{1}\rbr{\Xi}$ and that the pre-image of $L_{2}\rbr{\Xi}$
under $\op A_{\mu}^{0}$ is $H^{2}\rbr{\Xi}$; as a consequence, the
range of $\op K_{\mu}^{\varepsilon}$ is contained in~$H^{1}\rbr{\Xi}$.
\end{rem}

\begin{rem}
While $\rbr{\op A_{\mu}^{\varepsilon}}^{-1}$ and~$\op D_{2}\rbr{\op A_{\mu}^{\varepsilon}}^{-1}$
have limits, the operator~$\op D_{1}\rbr{\op A_{\mu}^{\varepsilon}}^{-1}$
can fail to converge, because, though the norm of $\varepsilon\op D_{1}\op K_{\mu}^{\varepsilon}$
is bounded uniformly in~$\varepsilon$, it need not go to zero.
However, if, for example, $\op D_{1}^{*}A=0$ and~$\op D_{1}^{*}a_{2}=0$
(in the weak sense), then $\op K_{\mu}^{\varepsilon}=0$, and $\op D_{1}\rbr{\op A_{\mu}^{\varepsilon}}^{-1}$
is therefore convergent as well. Notice that, in this case, the effective
coefficients are obtained by simply taking the  mean over~$\Omega_{1}$.

\end{rem}

\begin{rem}
We may replace $\op P^{\varepsilon}$ with another smoothing. For~instance,
the Steklov averaging operator (see~\cite{Zh}) or the scale-splitting
operator (see~\cite{Gr1}) can be used instead. This follows from
the inequalities
\begin{align}
\norm{\rbr{\op D_{1}N}w}_{2,\Xi}^{2} & \lesssim\norm{N\otimes\op D_{1}w}_{2,\Xi}^{2}+\norm w_{2,\Xi}^{2},\label{est: (D1 N) as a multiplier on Xi}\\
\norm{\rbr{\op D_{1}M}w}_{2,\Xi}^{2} & \lesssim\norm{M\op D_{1}w}_{2,\Xi}^{2}+\norm w_{1,2,\Xi}^{2},\label{est: (D1 M) as a multiplier on Xi}
\end{align}
which hold for any $w\in C_{0}^{\infty}\rbr{\Xi}$ (the~proof of
the inequalities is parallel to that of Lemma~\ref{lem: The auxiliary problem solution as multiplicators}),
and  properties of the smoothing operators (cf.~similar techniques
in \cite[Lemma~3.5]{PSu}). The~reason why we chose $\op P^{\varepsilon}$
is merely one of convenience: as we shall see below, $\op P^{\varepsilon}$
takes a rather simple form after passing to the fundamental domain.
\end{rem}

\begin{rem}
As~already indicated, the operator~$\op P^{\varepsilon}$ guarantees
that the range of $\op K_{\mu}^{\varepsilon}$ is contained in~$H^{1}\rbr{\Xi}$.
This means that, in general, it is not possible to  remove~$\op P^{\varepsilon}$.
However, this can be done in certain cases. For~example, if
$N\in\M\rbr{L_{2}\rbr{\Omega}}^{1\times d}$ and~$M\in\M\rbr{H^{1}\rbr{\Omega},L_{2}\rbr{\Omega}}$,
then the classical corrector
\[
\hat{\op K}_{\mu}^{\varepsilon}=\rbr{N^{\varepsilon}\op D+M^{\varepsilon}}\rbr{\op A_{\mu}^{0}}^{-1}
\]
as well as the composition~$\op D_{1}\hat{\op K}_{\mu}^{\varepsilon}$
are bounded on $L_{2}\rbr{\Xi}$ (by~(\ref{est: (D1 N) as a multiplier on Xi})
and~(\ref{est: (D1 M) as a multiplier on Xi})), and the estimates~(\ref{est: Approximation with 1st corrector})
and~(\ref{est: Approximation with 2d corrector}) remain true with
$\hat{\op K}_{\mu}^{\varepsilon}$ in place of $\op K_{\mu}^{\varepsilon}$
and
\[
\hat{\op C}_{\mu}^{\varepsilon}=\bigrbr{\hat{\op K}_{\mu}^{\varepsilon}-\op L_{\mu}}+\bigrbr{\rbr{\hat{\op K}_{\mu}^{\varepsilon}}^{+}-\op L_{\mu}^{+}}^{*}
\]
in~place of~$\op C_{\mu}^{\varepsilon}$.
\end{rem}

\section{\label{sec: Problem on the fundamental domain}Problem on the fundamental
domain}

Our strategy is to reduce $\op A_{\mu}^{\varepsilon}$ to an operator
on the fundamental domain~$\Omega$ and then formulate Theorems~\ref{thm: Convergence}\hairspace--\hairspace\ref{thm: Approximation with 2d corrector}
in terms of this latter operator.

Let~$\tau=\rbr{k,\varepsilon}\in\set T=\Omega_{1}^{*}\times\set E$.
We introduce the notation~$\op D_{1}\rbr{\tau}=\op D_{1}+\vect k$,
$\op D_{2}\rbr{\tau}=\varepsilon\op D_{2}$ and~$\op D\rbr{\tau}=\op D_{1}\rbr{\tau}+\op D_{2}\rbr{\tau}$
and set
\[
\norm u_{1,2,\Omega;\tau}=\bigrbr{\norm{\op D\rbr{\tau}u}_{2,\Omega}^{2}+\abs{\tau}^{2}\norm u_{2,\Omega}^{2}}^{1/2}
\]
and
\begin{align*}
\norm u_{1_{1},2,\Omega;\tau} & =\bigrbr{\norm{\op D_{1}\rbr{\tau}u}_{2,\Omega}^{2}+\abs{\tau}^{2}\norm u_{2,\Omega}^{2}}^{1/2},\\
\norm u_{1_{2},2,\Omega;\tau} & =\bigrbr{\norm{\op D_{2}\rbr{\tau}u}_{2,\Omega}^{2}+\abs{\tau}^{2}\norm u_{2,\Omega}^{2}}^{1/2}
\end{align*}
for~any $u$ for which the right-hand sides make sense.

Now let us define periodic Sobolev spaces over the interior of~$\Omega$.
Recall that we view $\Lambda$ as acting on $\Xi$ and $\Omega\subset\R^{d}$
is a fundamental domain for~$\Lambda$. The~space~$\widetilde W_{p}^{m}\rbr{\Omega}$
consists of all functions that are in $W_{p}^{m}\rbr{\Omega}$ and
that have periodic extensions in $W_{p}^{m}\rbr K$ for each compact
set~$K\subset\Xi$. Let $\widetilde W_{p,0}^{m}\rbr{\Omega}$ be the subspace
of functions in $\widetilde W_{p}^{m}\rbr{\Omega}$ with zero mean. As~usual,
$\widetilde H^{m}\rbr{\Omega}=\widetilde W_{2}^{m}\rbr{\Omega}$ and~$\widetilde H_{0}^{m}\rbr{\Omega}=\widetilde W_{2,0}^{m}\rbr{\Omega}$.

We define the  form~$\form a\rbr{\tau}$ on $\widetilde H^{1}\rbr{\Omega}$
by
\begin{equation}
\begin{aligned}\form a\rbr{\tau}\sbr u & =\rbr{A\op D\rbr{\tau}u,\op D\rbr{\tau}u}_{2,\Omega}+\varepsilon\rbr{\op D\rbr{\tau}u,a_{1}u}_{2,\Omega}\\
 & \quad+\varepsilon\rbr{a_{2}u,\op D\rbr{\tau}u}_{2,\Omega}+\varepsilon^{2}\rbr{qu,u}_{2,\Omega}.
\end{aligned}
\label{def: a(=0003C4)}
\end{equation}
Note that, when estimating $\norm{a_{n}u}_{2,\Omega}$, $n\in\Natto 2$,
and~$\abs{\rbr{qu,u}_{2,\Omega}}$, we can replace $u$ by $uv$
with~$v\rbr x=e^{i\abr{x_{1},k}}$. Hence,
\begin{align}
\varepsilon\norm{a_{n}u}_{2,\Omega} & \le\norm{a_{n}}_{\M}\bigrbr{\varepsilon^{2}\norm{\op D_{1}\rbr{\tau}u}_{2,\Omega}^{2}+\norm{\op D_{2}\rbr{\tau}u}_{2,\Omega}^{2}+\varepsilon^{2}\norm u_{2,\Omega}^{2}}^{1/2},\label{est: a=002099 on =0003A9 with =0003B5}\\
\varepsilon^{2}\abs{\rbr{qu,u}_{2,\Omega}} & \le\norm q_{\M}\bigrbr{\varepsilon^{2}\norm{\op D_{1}\rbr{\tau}u}_{2,\Omega}^{2}+\norm{\op D_{2}\rbr{\tau}u}_{2,\Omega}^{2}+\varepsilon^{2}\norm u_{2,\Omega}^{2}};\label{est: q on =0003A9 with =0003B5}
\end{align}
in~particular, this means that
\begin{align}
\varepsilon\norm{a_{n}u}_{2,\Omega} & \le\norm{a_{n}}_{\M}\bigrbr{\norm{\op D\rbr{\tau}u}_{2,\Omega}^{2}+\varepsilon^{2}\norm u_{2,\Omega}^{2}}^{1/2},\label{est: a=002099 on =0003A9 without =0003B5}\\
\varepsilon^{2}\abs{\rbr{qu,u}_{2,\Omega}} & \le\norm q_{\M}\bigrbr{\norm{\op D\rbr{\tau}u}_{2,\Omega}^{2}+\varepsilon^{2}\norm u_{2,\Omega}^{2}}\label{est: q on =0003A9 without =0003B5}
\end{align}
for~all~$u\in H^{1}\rbr{\Omega}$. Therefore, the same reasoning
as for $\form a^{\varepsilon}$ gives
\begin{equation}
\abs{\form a\rbr{\tau}\sbr{u,v}}\le C_{\flat}\norm u_{1,2,\Omega;\tau}\norm v_{1,2,\Omega;\tau},\qquad u,v\in\widetilde H^{1}\rbr{\Omega},\label{est: a(=0003C4) is bounded}
\end{equation}
and
\begin{equation}
\Re\form a\rbr{\tau}\sbr u\ge c_{*}\norm{\op D\rbr{\tau}u}_{2,\Omega}^{2}-\varepsilon^{2}c_{\natural}\norm u_{2,\Omega}^{2},\qquad u\in\widetilde H^{1}\rbr{\Omega}.\label{est: a(=0003C4) is coercive}
\end{equation}

Define $\op A_{\mu}\rbr{\tau}=\op A\rbr{\tau}-\varepsilon^{2}\mu\colon\widetilde H^{1}\rbr{\Omega}\to\widetilde H^{1}\rbr{\Omega}^{*}$
to be the operator associated with the form~$\form a_{\mu}\rbr{\tau}=\form a\rbr{\tau}-\varepsilon^{2}\mu$.
It~follows  that $\op A_{\mu}\rbr{\tau}$ is an isomorphism if~$\mu\notin\set S_{1}$.
\begin{lem}
\label{lem: Estimates for A(=0003C4)}For~any $\mu\notin\set S$
and~$\tau\in\set T$ we have
\begin{align*}
\norm{\rbr{\op A_{\mu}\rbr{\tau}}^{-1}}_{\B\rbr{L_{2}\rbr{\Omega}}} & \lesssim\abs{\tau}^{-2},\\
\norm{\op D\rbr{\tau}\rbr{\op A_{\mu}\rbr{\tau}}^{-1}}_{\B\rbr{L_{2}\rbr{\Omega}}^{d}} & \lesssim\abs{\tau}^{-1},\\
\norm{\op D\rbr{\tau}\rbr{\op A_{\mu}\rbr{\tau}}^{-1}\op D\rbr{\tau}}_{\B\rbr{L_{2}\rbr{\Omega}}^{d\times d}} & \lesssim1,\\
\norm{\op D\rbr{\tau}\op D_{2}\rbr{\tau}\rbr{\op A_{\mu}\rbr{\tau}}^{-1}}_{\B\rbr{L_{2}\rbr{\Omega}}^{d\times d}} & \lesssim1,
\end{align*}
where the constants depend on  $\mu$ and the multiplier norms of
the coefficients.\end{lem}
\begin{proof}
We  do the case~$\mu\in\set R$, where
\[
\set R=\bigcbr{z\in\C\colon\Re z<-c_{\natural}}.
\]
The~general case then follows by the resolvent identity.

Expanding $u\in\widetilde H^{1}\rbr{\Omega}$ in a Fourier series
\[
u\rbr x=\abs{\Omega_{1}}^{-1/2}\sum_{\lambda^{*}\in\Lambda^{*}}\hat{u}_{\lambda^{*}}\rbr{x_{2}}\,e^{-i\abr{x_{1},\lambda^{*}}},
\]
we find that, for all~$k\in\Omega_{1}^{*}$, 
\begin{equation}
\norm{\op D_{1}\rbr{\tau}u}_{2,\Omega}^{2}=\sum_{\lambda^{*}\in\Lambda^{*}}\abs{\lambda^{*}+k}^{2}\norm{\hat{u}_{\lambda^{*}}}_{2,\Omega_{2}}^{2}\ge\abs k^{2}\norm u_{2,\Omega}^{2},\label{est: D=002081(=0003C4) is invertible}
\end{equation}
which means that
\[
\abs{\tau}^{2}\norm u_{2,\Omega}^{2}\le\norm{\op D\rbr{\tau}u}_{2,\Omega}^{2}+\varepsilon^{2}\norm u_{2,\Omega}^{2}.
\]
Combining this with (\ref{est: a(=0003C4) is coercive}) gives the
first estimate. The~second is immediate from the first and~(\ref{est: a(=0003C4) is coercive}),
and the third follows at once from~(\ref{est: a(=0003C4) is coercive}).
It~remains to prove the last.

We shall use the classical technique of difference quotients. To~that
end, we introduce a little notation. Let $e_{2,m}$, $m\in\Natto{d_{2}}$,
be the unit vector along the $x_{2,m}$\nobreakdash-axis. If~$u$
is any function on $\Omega$, then we define the difference quotient~$D_{2,m}^{h}$
in the variable~$x_{2,m}$ of size~$h\in\R\setminus\cbr 0$ by
setting $D_{2,m}^{h}u=-ih^{-1}\rbr{\op T_{2,m}^{h}u-u}$ where $\rbr{\op T_{2,m}^{h}u}\rbr x=u\rbr{x+he_{2,m}}$.
Observe that
\[
\rbr{D_{2,m}^{h}}^{*}=D_{2,m}^{-h}
\]
and
\[
D_{2,m}^{h}\rbr{uv}=\rbr{D_{2,m}^{h}u}\op T_{2,m}^{h}v+u\rbr{D_{2,m}^{h}v}.
\]

For~$f\in L_{2}\rbr{\Omega}$ fixed, we set~$w=\rbr{\op A_{\mu}\rbr{\tau}}^{-1}f$.
Then
\begin{equation}
\form a_{\mu}\rbr{\tau}\sbr{D_{2,m}^{h}w}=\rbr{f,D_{2,m}^{-h}D_{2,m}^{h}w}_{2,\Omega}-\bigrbr{\sbr{D_{2,m}^{h},\op A_{\mu}\rbr{\tau}}w,D_{2,m}^{h}w}_{2,\Omega}.\label{eq: Elliptic regularity equation}
\end{equation}
If~we show that
\begin{equation}
\bigabs{\bigrbr{\sbr{D_{2,m}^{h},\op A_{\mu}\rbr{\tau}}u,v}_{2,\Omega}}\lesssim\norm u_{1,2,\Omega;\tau}\norm v_{1,2,\Omega;\tau},\qquad u,v\in\widetilde H^{1}\rbr{\Omega},\label{est: Commutator of D=002082,=002098 and A(=0003C4) is bounded}
\end{equation}
where the constant is independent of~$h$, then, by the estimates
that we  just proved, the right-hand side of (\ref{eq: Elliptic regularity equation})
will not exceed
\[
\begin{aligned}\hspace{2em} & \hspace{-2em}\varepsilon^{-1}\norm{\op D\rbr{\tau}D_{2,m}^{h}w}_{2,\Omega}\norm f_{2,\Omega}+\abs{\tau}^{-1}C\norm{D_{2,m}^{h}w}_{1,2,\Omega;\tau}\norm f_{2,\Omega}\\
 & \le2^{-1}c_{*}\norm{\op D\rbr{\tau}D_{2,m}^{h}w}_{2,\Omega}^{2}+\varepsilon^{2}\abs{c_{\natural}+\Re\mu}\,\norm{D_{2,m}^{h}w}_{2,\Omega}^{2}+\varepsilon^{-2}C\norm f_{2,\Omega}^{2}
\end{aligned}
\]
with~some constant~$C$. We have used here the fact that, for~any~$h$,
\[
\norm{D_{2,m}^{-h}D_{2,m}^{h}w}_{2,\Omega}\le\norm{D_{2,m}D_{2,m}^{h}w}_{2,\Omega}.
\]
On~the~other hand, since  $-\rbr{c_{\natural}+\Re\mu}>0$, it follows
from (\ref{est: a(=0003C4) is coercive}) that
\[
\Re\form a_{\mu}\rbr{\tau}\sbr{D_{2,m}^{h}w}\ge c_{*}\norm{\op D\rbr{\tau}D_{2,m}^{h}w}_{2,\Omega}^{2}+\varepsilon^{2}\abs{c_{\natural}+\Re\mu}\,\norm{D_{2,m}^{h}w}_{2,\Omega}^{2}.
\]
As~a~result,
\[
\norm{\op D\rbr{\tau}\varepsilon D_{2,m}^{h}w}_{2,\Omega}\lesssim\norm f_{2,\Omega}
\]
uniformly in~$h$. Thus, there exists $\op D\rbr{\tau}\op D_{2,m}\rbr{\tau}w$,
with 
\[
\norm{\op D\rbr{\tau}\op D_{2,m}\rbr{\tau}w}_{2,\Omega}\lesssim\norm f_{2,\Omega},
\]
as~desired.

We conclude by proving~(\ref{est: Commutator of D=002082,=002098 and A(=0003C4) is bounded}).
Since 
\[
\begin{aligned}\bigrbr{\sbr{D_{2,m}^{h},\op A_{\mu}\rbr{\tau}}\op T_{2,m}^{-h}u,v}_{2,\Omega} & =\bigrbr{\rbr{D_{2,m}^{h}A}\op D\rbr{\tau}u,\op D\rbr{\tau}v}_{2,\Omega}-\varepsilon\bigrbr{\op D\rbr{\tau}u,\rbr{D_{2,m}^{h}a_{1}}v}_{2,\Omega}\\
 & \quad+\varepsilon\bigrbr{\rbr{D_{2,m}^{h}a_{2}}u,\op D\rbr{\tau}v}_{2,\Omega}+\varepsilon^{2}\bigrbr{\rbr{D_{2,m}^{h}q}u,v}_{2,\Omega},
\end{aligned}
\]
we see that it suffices to show that each coefficient of this form
is still a multiplier with norm bounded by a constant independent
of~$h$, because then an argument similar to the one we used for
proving (\ref{est: a(=0003C4) is bounded}) will lead to~(\ref{est: Commutator of D=002082,=002098 and A(=0003C4) is bounded}).
Obviously,
\[
\rbr{D_{2,m}^{h}\gamma}\rbr x=\int_{\rbr{0,1}}\rbr{\op T_{2,m}^{th}D_{2,m}\gamma}\rbr x\,dt
\]
for~every  function~$\gamma$ that has a derivative~$D_{2,m}\gamma\in L_{2}\rbr{\Omega}$.
A~duality argument shows that if $D_{2,m}\gamma\in H^{1}\rbr{\Omega}^{*}$,
then
\[
\bigrbr{\rbr{D_{2,m}^{h}\gamma}u,u}_{2,\Omega}=\int_{\rbr{0,1}}\bigrbr{\rbr{D_{2,m}\gamma}\op T_{2,m}^{-th}u,\op T_{2,m}^{-th}u}_{2,\Omega}dt,
\]
any~$u\in C^{1}\rbr{\bar{\Omega}}$. Hence, $\norm{D_{2,m}^{h}A}_{\M}\le\norm{D_{2,m}A}_{\M}$,
$\norm{D_{2,m}^{h}a_{n}}_{\M}\le\norm{D_{2,m}a_{n}}_{\M}$, $n\in\Natto 2$,
and~$\norm{D_{2,m}^{h}q}_{\M}\le\norm{D_{2,m}q}_{\M}$. This  completes
the proof.\qedspace
\end{proof}
Let $\op A_{\mu}\rbr{\tau}^{+}\colon\widetilde H^{1}\rbr{\Omega}\to\widetilde H^{1}\rbr{\Omega}^{*}$
be the the formal adjoint of~$\op A_{\mu}\rbr{\tau}$. From~(\ref{est: a(=0003C4) is bounded})
and~(\ref{est: a(=0003C4) is coercive}), we see that $\op A_{\mu}\rbr{\tau}^{+}$
is also an isomorphism whenever~$\mu\notin\set S_{1}$. Moreover,
 the conclusion of Lemma~\ref{lem: Estimates for A(=0003C4)} holds
for~$\op A_{\mu}\rbr{\tau}^{+}$. It~is easy to note the relationship
between $\op A_{\mu}\rbr{\tau}$ and~$\op A_{\mu}\rbr{\tau}^{+}$.
Indeed, a suitable restriction of $\op A_{\mu}\rbr{\tau}^{+}$ is
the adjoint of the restriction of~$\op A_{\mu}\rbr{\tau}$, so that,
if $\mu\notin\set S_{1}$,
\begin{equation}
\bigrbr{\rbr{\op A_{\mu}\rbr{\tau}}^{-1}f,f}_{2,\Omega}=\bigrbr{f,\rbr{\op A_{\mu}\rbr{\tau}^{+}}^{-1}f}_{2,\Omega},\qquad f\in L_{2}\rbr{\Omega}.\label{eq: Relation for A(=0003C4) and A(=0003C4)=00207A}
\end{equation}

We shall think of $L_{2}\rbr{\Omega}$ as the tensor product~$L_{2}\rbr{\Omega_{1}}\otimes L_{2}\rbr{\Omega_{2}}$.
Recall that $\op G$ is the Gelfand transform and $\op S^{\varepsilon}$
is the scaling transformation. Clearly, $\op G\op S^{\varepsilon}\otimes\op I$
maps $H^{1}\rbr{\Xi}$ onto $L_{2}\rbr{\Omega_{1}^{*};\widetilde H^{1}\rbr{\Omega}}$
and, for any $u\in H^{1}\rbr{\Xi}$,
\[
\form a_{\mu}^{\varepsilon}\sbr u=\varepsilon^{-2}\int_{\Omega_{1}^{*}}\form a_{\mu}\rbr{\tau}\sbr{\tilde{u}\rbr{k,\wc}}\,dk,
\]
where~$\tilde{u}=\rbr{\op G\op S^{\varepsilon}\otimes\op I}u$; that
is,
\begin{equation}
\rbr{\op G\op S^{\varepsilon}\otimes\op I}\rbr{\op A_{\mu}^{\varepsilon}}^{-1}\rbr{\op G\op S^{\varepsilon}\otimes\op I}^{-1}=\int_{\Omega_{1}^{*}}^{\oplus}\varepsilon^{2}\rbr{\op A_{\mu}\rbr{\tau}}^{-1}dk.\label{eq: Direct integral for A=001D4B}
\end{equation}

Now, we would like to do the same for the operator~$\op A_{\mu}^{0}$.
To~this end, let $\form a^{0}\rbr{\tau}$ be the form on $\widetilde H^{1}\rbr{\Omega}$
defined by 
\begin{equation}
\begin{aligned}\form a^{0}\rbr{\tau}\sbr u & =\rbr{A^{0}\op D\rbr{\tau}u,\op D\rbr{\tau}u}_{2,\Omega}+\varepsilon\rbr{\op D\rbr{\tau}u,a_{1}^{0}u}_{2,\Omega}\\
 & \quad+\varepsilon\rbr{a_{2}^{0}u,\op D\rbr{\tau}u}_{2,\Omega}+\varepsilon^{2}\rbr{q^{0}u,u}_{2,\Omega}
\end{aligned}
\label{def: a=002070(=0003C4)}
\end{equation}
where
\[
\rbr{q^{0}u,u}_{2,\Omega}=\int_{\Omega_{1}}\rbr{q^{0}u\rbr{x_{1},\wc},u\rbr{x_{1},\wc}}_{2,\Omega_{2}}dx_{1}.
\]
The~same arguments used to obtain (\ref{est: a=002070 is bounded})
and~(\ref{est: a=002070 is coercive}) now show that 
\begin{equation}
\abs{\form a^{0}\rbr{\tau}\sbr{u,v}}\le C_{\flat}^{0}\norm u_{1,2,\Omega;\tau}\norm v_{1,2,\Omega;\tau},\qquad u,v\in\widetilde H^{1}\rbr{\Omega},\label{est: a=002070(=0003C4) is bounded}
\end{equation}
and
\begin{equation}
\Re\form a^{0}\rbr{\tau}\sbr u\ge c_{*}\norm{\op D\rbr{\tau}u}_{2,\Omega}^{2}-\varepsilon^{2}c_{\natural}\norm u_{2,\Omega}^{2},\qquad u\in\widetilde H^{1}\rbr{\Omega}.\label{est: a=002070(=0003C4) is coercive}
\end{equation}
Let $\op A_{\mu}^{0}\rbr{\tau}=\op A^{0}\rbr{\tau}-\varepsilon^{2}\mu\colon\widetilde H^{1}\rbr{\Omega}\to\widetilde H^{1}\rbr{\Omega}^{*}$
be the operator corresponding to~$\form a_{\mu}^{0}\rbr{\tau}=\form a^{0}\rbr{\tau}-\varepsilon^{2}\mu$.
Then $\op A_{\mu}^{0}\rbr{\tau}$ is an isomorphism if~$\mu\notin\set S_{0}$.
\begin{lem}
\label{lem: Estimates for A=002070(=0003C4)}For~any $\mu\notin\set S$
and~$\tau\in\set T$ we have
\begin{align*}
\norm{\rbr{\op A_{\mu}^{0}\rbr{\tau}}^{-1}}_{\B\rbr{L_{2}\rbr{\Omega}}} & \lesssim\abs{\tau}^{-2},\\
\norm{\op D\rbr{\tau}\rbr{\op A_{\mu}^{0}\rbr{\tau}}^{-1}}_{\B\rbr{L_{2}\rbr{\Omega}}^{d}} & \lesssim\abs{\tau}^{-1},\\
\norm{\op D\rbr{\tau}\rbr{\op A_{\mu}^{0}\rbr{\tau}}^{-1}\op D\rbr{\tau}}_{\B\rbr{L_{2}\rbr{\Omega}}^{d\times d}} & \lesssim1,\\
\norm{\op D\rbr{\tau}\op D\rbr{\tau}\rbr{\op A_{\mu}^{0}\rbr{\tau}}^{-1}}_{\B\rbr{L_{2}\rbr{\Omega}}^{d\times d}} & \lesssim1,
\end{align*}
where the constants depend on  $\mu$ and the multiplier norms of
the coefficients.\end{lem}
\begin{proof}
The~proof is similar to that of Lemma~\ref{lem: Estimates for A(=0003C4)}.\qedspace
\end{proof}
Since, for every~$u\in H^{1}\rbr{\Xi}$,
\[
\form a_{\mu}^{0}\sbr u=\varepsilon^{-2}\int_{\Omega_{1}^{*}}\form a_{\mu}^{0}\rbr{\tau}\sbr{\tilde{u}\rbr{k,\wc}}\,dk,
\]
with~$\tilde{u}=\rbr{\op G\op S^{\varepsilon}\otimes\op I}u$, it
follows that
\begin{equation}
\rbr{\op G\op S^{\varepsilon}\otimes\op I}\rbr{\op A_{\mu}^{0}}^{-1}\rbr{\op G\op S^{\varepsilon}\otimes\op I}^{-1}=\int_{\Omega_{1}^{*}}^{\oplus}\varepsilon^{2}\rbr{\op A_{\mu}^{0}\rbr{\tau}}^{-1}dk.\label{eq: Direct integral for A=002070}
\end{equation}
As~a~result, we may restate Theorem~\ref{thm: Convergence} in
terms of the fibers~$\op A_{\mu}\rbr{\tau}$ and~$\op A_{\mu}^{0}\rbr{\tau}$.
\begin{thm}
\label{thm: Convergence on =0003A9}Let~$\mu\notin\set S$. Then
for all $\tau\in\set T$ it holds that 
\begin{align*}
\norm{\rbr{\op A_{\mu}\rbr{\tau}}^{-1}-\rbr{\op A_{\mu}^{0}\rbr{\tau}}^{-1}}_{\B\rbr{L_{2}\rbr{\Omega}}} & \lesssim\abs{\tau}^{-1},\\
\norm{\op D_{2}\rbr{\tau}\rbr{\op A_{\mu}\rbr{\tau}}^{-1}-\op D_{2}\rbr{\tau}\rbr{\op A_{\mu}^{0}\rbr{\tau}}^{-1}}_{\B\rbr{L_{2}\rbr{\Omega}}^{d}} & \lesssim1,
\end{align*}
where the constants depend only on $r_{\Lambda}$, $\mu$ and the
multiplier norms of the coefficients.
\end{thm}
Let $\op P_{1}$ and~$\op P_{2}$ denote the orthogonal projections
in $L_{2}\rbr{\Omega}$ onto $\C\otimes L_{2}\rbr{\Omega_{2}}$ and~$L_{2}\rbr{\Omega_{1}}\otimes\C$,
respectively. Notice that
\[
\rbr{\op G\op S^{\varepsilon}\otimes\op I}\op P^{\varepsilon}\rbr{\op G\op S^{\varepsilon}\otimes\op I}^{-1}=\int_{\Omega_{1}^{*}}^{\oplus}\op P_{1}\,dk.
\]
Define $\op K_{\mu}\rbr{\tau}\colon L_{2}\rbr{\Omega}\to\widetilde H^{1}\rbr{\Omega}$~by
\begin{equation}
\op K_{\mu}\rbr{\tau}=\rbr{N\op D\rbr{\tau}+\varepsilon M}\rbr{\op A_{\mu}^{0}\rbr{\tau}}^{-1}\op P_{1}.\label{def: K(=0003C4)}
\end{equation}

\begin{lem}
\label{lem: Estimates for K(=0003C4)}For~any $\mu\notin\set S$
and~$\tau\in\set T$ we have
\begin{align*}
\norm{\op D_{1}\op K_{\mu}\rbr{\tau}}_{\B\rbr{L_{2}\rbr{\Omega}}^{d}} & \lesssim\abs{\tau}^{-1},\\
\norm{\op D_{1}\op D_{2}\rbr{\tau}\op K_{\mu}\rbr{\tau}}_{\B\rbr{L_{2}\rbr{\Omega}}^{d\times d}} & \lesssim1,
\end{align*}
where the constants depend on  $\mu$ and the multiplier norms of
the coefficients.\end{lem}
\begin{proof}
Let~$f\in L_{2}\rbr{\Omega}$, and let $u=\rbr{\op A_{\mu}^{0}\rbr{\tau}}^{-1}\op P_{1}f$
and~$U=\op K_{\mu}\rbr{\tau}f$. Then, by Lemma~\ref{lem: Estimates for A=002070(=0003C4)},
it follows that
\[
\norm{\op D_{1}U}_{2,\Omega}\le\abs{\Omega_{1}}^{-1/2}\bigrbr{\norm{\op D_{1}N}_{\M}+\norm{\op D_{1}M}_{\M}}\norm u_{1_{2},2,\Omega;\tau}\lesssim\abs{\tau}^{-1}\norm f_{2,\Omega}
\]
and
\[
\begin{aligned}\norm{\op D_{1}\op D_{2}\rbr{\tau}U}_{2,\Omega} & \le\abs{\Omega_{1}}^{-1/2}\bigrbr{\norm{\op D_{1}\op D_{2}N}_{\M}+\norm{\op D_{1}\op D_{2}M}_{\M}}\abs{\tau}\norm u_{1_{2},2,\Omega;\tau}\\
 & \quad+\abs{\Omega_{1}}^{-1/2}\bigrbr{\norm{\op D_{1}N}_{\M}+\norm{\op D_{1}M}_{\M}}\norm{\op D_{2}\rbr{\tau}u}_{1_{2},2,\Omega;\tau}\lesssim\norm f_{2,\Omega},
\end{aligned}
\]
as~required.\qedspace
\end{proof}
We remark that, since $\op P_{1}\op K_{\mu}\rbr{\tau}=0$ (by~the~definitions
of $N$ and~$M$), we may use Poincar\'{e}'s inequality to see that
$\op K_{\mu}\rbr{\tau}$ and~$\op D_{2}\rbr{\tau}\op K_{\mu}\rbr{\tau}$
satisfy estimates similar to those for $\op D_{1}\op K_{\mu}\rbr{\tau}$
and~$\op D_{1}\op D_{2}\rbr{\tau}\op K_{\mu}\rbr{\tau}$, respectively.
This means that, unlike the case of $\op A_{\mu}\rbr{\tau}$ and~$\op A_{\mu}^{0}\rbr{\tau}$,
where both $\op D_{1}\rbr{\tau}$ and~$\op D_{2}\rbr{\tau}$ make
the norms of the corresponding compositions smaller, roughly speaking,
by multiplying each of these norms by $\abs{\tau}$, the differentiation~$\op D_{1}$
will not change the order of the norm of~$\op K_{\mu}\rbr{\tau}$.
The~reason, of course, is that the corrector~$\op K_{\mu}^{\varepsilon}$
involves functions that rapidly oscillate in the first variable.

In~the~same fashion as above, we may prove that
\begin{equation}
\rbr{\op G\op S^{\varepsilon}\otimes\op I}\op K_{\mu}^{\varepsilon}\rbr{\op G\op S^{\varepsilon}\otimes\op I}^{-1}=\int_{\Omega_{1}^{*}}^{\oplus}\varepsilon\op K_{\mu}\rbr{\tau}\,dk.\label{eq: Direct integral for K=001D4B}
\end{equation}
Theorem~\ref{thm: Approximation with 1st corrector} now takes the
following form:
\begin{thm}
\label{thm: Approximation on =0003A9 with 1st corrector}Let~$\mu\notin\set S$.
Then for all $\tau\in\set T$ it holds that 
\[
\norm{\op D_{1}\rbr{\tau}\rbr{\op A_{\mu}\rbr{\tau}}^{-1}-\op D_{1}\rbr{\tau}\rbr{\op A_{\mu}^{0}\rbr{\tau}}^{-1}-\op D_{1}\rbr{\tau}\op K_{\mu}\rbr{\tau}}_{\B\rbr{L_{2}\rbr{\Omega}}^{d}}\lesssim1,
\]
where the constant depends only on $r_{\Lambda}$, $\mu$ and the
multiplier norms of the coefficients.
\end{thm}
Let $\op S\rbr{\tau}\colon\widetilde H^{1}\rbr{\Omega}\to\widetilde H^{1}\rbr{\Omega}^{*}$
and~$\op T\rbr{\tau}\colon\widetilde H^{1}\rbr{\Omega}\to\widetilde H^{1}\rbr{\Omega}^{*}$
be given by 
\begin{align}
\op S\rbr{\tau} & =\bigrbr{\rbr{\vect k+\op D_{2}\rbr{\tau}}^{*}A+\varepsilon a_{1}^{*}}\rbr{\vect k+\op D_{2}\rbr{\tau}}+\varepsilon\rbr{\vect k+\op D_{2}\rbr{\tau}}^{*}a_{2}+\varepsilon^{2}q,\label{def: S(=0003C4)}\\
\op T\rbr{\tau} & =\bigrbr{\rbr{\vect k+\op D_{2}\rbr{\tau}}^{*}A+\varepsilon a_{1}^{*}}\op D_{1}.\label{def: T(=0003C4)}
\end{align}
Clearly, $\op S\rbr{\tau}$ and~$\op T\rbr{\tau}$ are bounded operators,
satisfying estimates   like that for~$\op A_{\mu}\rbr{\tau}$.
It~is in fact possible to improve these estimates by using (\ref{est: a=002099 on =0003A9 with =0003B5})
and~(\ref{est: q on =0003A9 with =0003B5}) instead of (\ref{est: a=002099 on =0003A9 without =0003B5})
and~(\ref{est: q on =0003A9 without =0003B5}):
\begin{align}
\abs{\rbr{\op S\rbr{\tau}u,v}_{2,\Omega}} & \lesssim\bigrbr{\varepsilon\norm{\op D_{1}\rbr{\tau}u}_{2,\Omega}+\norm u_{1_{2},2,\Omega;\tau}}\bigrbr{\varepsilon\norm{\op D_{1}\rbr{\tau}v}_{2,\Omega}+\norm v_{1_{2},2,\Omega;\tau}},\label{est: S(=0003C4)}\\
\abs{\rbr{\op T\rbr{\tau}u,v}_{2,\Omega}} & \lesssim\norm{\op D_{1}u}_{2,\Omega}\bigrbr{\varepsilon\norm{\op D_{1}\rbr{\tau}v}_{2,\Omega}+\norm v_{1_{2},2,\Omega;\tau}}\label{est: T(=0003C4)}
\end{align}
if~$u,v\in\widetilde H^{1}\rbr{\Omega}$. The~operators~$\op S\rbr{\tau}^{+}$
and~$\op T\rbr{\tau}^{+}$ are defined likewise. Of~course, estimates
similar to (\ref{est: S(=0003C4)}) and~(\ref{est: T(=0003C4)})
hold for $\op S\rbr{\tau}^{+}$ and~$\op T\rbr{\tau}^{+}$ as~well.
Notice that
\begin{equation}
\op A_{\mu}\rbr{\tau}=\op D_{1}^{*}A\op D_{1}+\op S\rbr{\tau}+\op T\rbr{\tau}+\rbr{\op T\rbr{\tau}^{+}}^{*}-\varepsilon^{2}\mu.\label{eq: Expression for A(=0003C4) in terms of S(=0003C4) and T(=0003C4)}
\end{equation}

We break  Theorem~\ref{thm: Approximation with 2d corrector}
into two parts. The~first is formulated as follows:
\begin{lem}
For~any $\mu\notin\set S$ and~$\varepsilon\in\set E$ we have
\[
\norm{\op L_{\mu}\rbr{\op I-\op P^{\varepsilon}}}_{\B\rbr{L_{2}\rbr{\Xi}}}\lesssim\varepsilon,
\]
where the constant depends on $r_{\Lambda}$, $\mu$ and the multiplier
norms of the coefficients.\end{lem}
\begin{proof}
We estimate the operator norm of the symbol~$\op L_{\mu}\rbr{\wc}$.
Fix~$k\in\R^{d_{1}}\setminus\cbr 0$. Let $f\in L_{2}\rbr{\Omega_{2}}$,
and let $u=\rbr{\op A_{\mu}^{0}\rbr{\tau}}^{-1}f$ and~$U=\op K_{\mu}\rbr{\tau}f$,
$U^{+}=\op K_{\mu}\rbr{\tau}^{+}f$ where~$\tau=\rbr{k,1}$. Then
\[
\rbr{\op L_{\mu}\rbr kf,f}_{2,\Omega_{2}}=\abs{\Omega_{1}}^{-1}\bigrbr{\op S\rbr{\tau}u+\op T\rbr{\tau}U,U^{+}}_{2,\Omega}.
\]
It~should be noted that while Lemmas~\ref{lem: Estimates for A=002070(=0003C4)}
and~\ref{lem: Estimates for K(=0003C4)} are only asserted to be
valid for $\op A_{\mu}^{0}\rbr{\tau}$ and~$\op K_{\mu}\rbr{\tau}$
with~$\tau\in\Omega_{1}^{*}\times\set E$, they may be extended to
$\op A_{\mu}^{0}\rbr{\tau}\op P_{1}$ and~$\op K_{\mu}\rbr{\tau}$
with~$\tau\in\R^{d_{1}}\times\set E$. Indeed, the condition~$k\in\Omega_{1}^{*}$
is used only  to ensure the inequality~(\ref{est: D=002081(=0003C4) is invertible}).
But when $u$ does not depend on~$x_{1}$, we have equality for
each~$k\in\R^{d_{1}}$. Thus, the estimates~(\ref{est: S(=0003C4)})
and~(\ref{est: T(=0003C4)}) together with these extended versions
of Lemmas~\ref{lem: Estimates for A=002070(=0003C4)} and~\ref{lem: Estimates for K(=0003C4)},
as well as Poincar\'{e}'s inequality, give
\[
\abs{\rbr{\op L_{\mu}\rbr kf,f}_{2,\Omega_{2}}}\lesssim\bigrbr{\norm u_{1_{2},2,\Omega;\tau}+\norm{\op D_{1}U}_{2,\Omega}}\norm{\op D_{1}U^{+}}_{1_{2},2,\Omega;\tau}\lesssim\abs k^{-1}\norm f_{2,\Omega_{2}}^{2}.
\]
Now if $g\in C_{0}^{\infty}\rbr{\Xi}$ and $\hat{g}=\rbr{\op F\otimes\op I}g$,
then
\[
\begin{aligned}\abs{\rbr{\op L_{\mu}\rbr{\op I-\op P^{\varepsilon}}g,g}_{2,\Xi}} & \le\int_{\R^{d_{1}}\setminus\varepsilon^{-1}\Omega_{1}^{*}}\bigabs{\bigrbr{\op L_{\mu}\rbr k\hat{g}\rbr{k,\wc},\hat{g}\rbr{k,\wc}}_{2,\Omega_{2}}}\,dk\\
 & \lesssim\int_{\R^{d_{1}}\setminus\varepsilon^{-1}\Omega_{1}^{*}}\abs k^{-1}\norm{\hat{g}\rbr{k,\wc}}_{2,\Omega}^{2}\,dk\le\varepsilon r_{\Lambda}^{-1}\norm g_{2,\Xi}^{2}.
\end{aligned}
\]
This is the result that we wished to prove.\qedspace
\end{proof}
The~lemma takes care of  $\op L_{\mu}\rbr{\op I-\op P^{\varepsilon}}$
and~$\op L_{\mu}^{+}\rbr{\op I-\op P^{\varepsilon}}$ in the estimate~(\ref{est: Approximation with 2d corrector}),
so we may concentrate our attention on $\op L_{\mu}\op P^{\varepsilon}$
and~$\op L_{\mu}^{+}\op P^{\varepsilon}$. Let $\op A_{\mu}^{0}\rbr{\tau}^{+}$
and~$\op K_{\mu}\rbr{\tau}^{+}$ play the roles of $\op A_{\mu}^{0}\rbr{\tau}$
and~$\op K_{\mu}\rbr{\tau}$ for~$\op A_{\mu}\rbr{\tau}^{+}$.
Then, since
\[
\bigrbr{\rbr{\op G\op S^{\varepsilon}\otimes\op I}\op P^{\varepsilon}u}\rbr{k,x}=\abs{\Omega_{1}}^{-1/2}\bigrbr{\rbr{\op S^{1/\varepsilon}\op F\otimes\op I}u}\rbr{k,x_{2}}
\]
for~every~$u\in L_{2}\rbr{\Xi}$, we find that
\begin{eqnarray}
\rbr{\op G\op S^{\varepsilon}\otimes\op I}\op L_{\mu}\op P^{\varepsilon}\rbr{\op G\op S^{\varepsilon}\otimes\op I}^{-1} & = & \int_{\Omega_{1}^{*}}^{\oplus}\varepsilon\op L_{\mu}\rbr{\tau}\op P_{1}\,dk,\label{eq: Direct integral for L}
\end{eqnarray}
where~$\op L_{\mu}\rbr{\tau}\colon L_{2}\rbr{\Omega}\to L_{2}\rbr{\Omega}$
is given by
\begin{equation}
\op L_{\mu}\rbr{\tau}=\bigrbr{\op K_{\mu}\rbr{\tau}^{+}}^{*}\bigrbr{\op S\rbr{\tau}\rbr{\op A_{\mu}^{0}\rbr{\tau}}^{-1}+\op T\rbr{\tau}\op K_{\mu}\rbr{\tau}}.\label{def: L(=0003C4)}
\end{equation}
Define~$\op L_{\mu}\rbr{\tau}^{+}$ similarly. Obviously, in order
to prove Theorem~\ref{thm: Approximation with 2d corrector}, we
need to establish the following result:
\begin{thm}
\label{thm: Approximation on =0003A9 with 2d corrector}Let~$\mu\notin\set S$.
Then for all $\tau\in\set T$ it holds that 
\[
\bignorm{\rbr{\op A_{\mu}\rbr{\tau}}^{-1}-\rbr{\op A_{\mu}^{0}\rbr{\tau}}^{-1}-\bigrbr{\op K_{\mu}\rbr{\tau}-\op L_{\mu}\rbr{\tau}}\op P_{1}-\op P_{1}\bigrbr{\op K_{\mu}\rbr{\tau}^{+}-\op L_{\mu}\rbr{\tau}^{+}}^{*}}_{\B\rbr{L_{2}\rbr{\Omega}}}\lesssim1,
\]
where the constant depends only on $r_{\Lambda}$, $\mu$ and the
multiplier norms of the coefficients.
\end{thm}
We now turn to the proofs. Our first goal is to verify the identity
\begin{equation}
\begin{aligned}\hspace{2em} & \hspace{-2em}\rbr{\op A_{\mu}\rbr{\tau}}^{-1}\op P_{1}-\rbr{\op A_{\mu}^{0}\rbr{\tau}}^{-1}\op P_{1}-\op K_{\mu}\rbr{\tau}\\
 & =-\rbr{\op A_{\mu}\rbr{\tau}}^{-1}\op P_{1}^{\bot}\bigrbr{\op S\rbr{\tau}\rbr{\op A_{\mu}^{0}\rbr{\tau}}^{-1}\op P_{1}+\op T\rbr{\tau}\op K_{\mu}\rbr{\tau}}\\
 & \quad-\rbr{\op A_{\mu}\rbr{\tau}}^{-1}\bigrbr{\op S\rbr{\tau}^{+}+\op T\rbr{\tau}^{+}-\varepsilon^{2}\bar{\mu}}^{*}\op K_{\mu}\rbr{\tau}.
\end{aligned}
\label{eq: Identity for U(=0003C4)}
\end{equation}
Denote the operator on the left by~$\op U_{\mu}\rbr{\tau}$. If~$f,g\in L_{2}\rbr{\Omega}$,
then we set $u=\rbr{\op A_{\mu}^{0}\rbr{\tau}}^{-1}f$, $U=\op K_{\mu}\rbr{\tau}f$
and~$v^{+}=\rbr{\op A_{\mu}\rbr{\tau}^{+}}^{-1}g$. By~(\ref{eq: Relation for A(=0003C4) and A(=0003C4)=00207A}),
we have
\[
\rbr{\op U_{\mu}\rbr{\tau}f,g}_{2,\Omega}=\form a_{\mu}^{0}\rbr{\tau}\sbr{\op P_{1}u,v^{+}}-\form a_{\mu}\rbr{\tau}\sbr{\op P_{1}u+U,v^{+}}
\]
(and here we are using the fact that $\op P_{1}$ commutes with~$\op D\rbr{\tau}$
on periodic functions). Looking at the definitions of the effective
coefficients, we see that
\begin{equation}
\op A_{\mu}^{0}\rbr{\tau}\op P_{1}=\op P_{1}\bigrbr{\op A_{\mu}\rbr{\tau}+\op T\rbr{\tau}\rbr{N\op D\rbr{\tau}+\varepsilon M}}\op P_{1},\label{eq: Identity for A=002070(=0003C4)P=002081}
\end{equation}
from~which we obtain
\begin{equation}
\form a_{\mu}^{0}\rbr{\tau}\sbr{\op P_{1}u,v^{+}}-\form a_{\mu}\rbr{\tau}\sbr{\op P_{1}u,v^{+}}=\rbr{\op T\rbr{\tau}U,\op P_{1}v^{+}}_{2,\Omega}-\rbr{\op A_{\mu}\rbr{\tau}\op P_{1}u,\op P_{1}^{\bot}v^{+}}_{2,\Omega}.\label{eq: a=002070(=0003C4)=00005BP=002081u,v=00207A=00005D-a(=0003C4)=00005BP=002081u,v=00207A=00005D}
\end{equation}
On~the~other hand, it follows from~(\ref{eq: Expression for A(=0003C4) in terms of S(=0003C4) and T(=0003C4)})
that
\begin{equation}
\form a_{\mu}\rbr{\tau}\sbr{U,v^{+}}=\rbr{\op T\rbr{\tau}U,\op P_{1}v^{+}}_{2,\Omega}+\rbr{\op S\rbr{\tau}U,\op P_{1}v^{+}}_{2,\Omega}+\rbr{\op A_{\mu}\rbr{\tau}U,\op P_{1}^{\bot}v^{+}}_{2,\Omega}.\label{eq: a(=0003C4)=00005BU,v=00207A=00005D}
\end{equation}
The~first term on the right-hand side of this last equality cancels
with the first term on the right-hand side of~(\ref{eq: a=002070(=0003C4)=00005BP=002081u,v=00207A=00005D-a(=0003C4)=00005BP=002081u,v=00207A=00005D}),
so
\[
\rbr{\op U_{\mu}\rbr{\tau}f,g}_{2,\Omega}=-\rbr{\op A_{\mu}\rbr{\tau}\rbr{\op P_{1}u+U},\op P_{1}^{\bot}v^{+}}_{2,\Omega}-\rbr{\op S\rbr{\tau}U,\op P_{1}v^{+}}_{2,\Omega}.
\]
Notice that, by the definitions of $N$ and~$M$,
\[
\op D_{1}^{*}A\op D_{1}U+\rbr{\op T\rbr{\tau}^{+}}^{*}\op P_{1}u=0.
\]
 Since $\op T\rbr{\tau}\op P_{1}=0$ and~$\op P_{1}\op K_{\mu}\rbr{\tau}=0$,
this and the identity~(\ref{eq: Expression for A(=0003C4) in terms of S(=0003C4) and T(=0003C4)})
imply that
\[
\rbr{\op U_{\mu}\rbr{\tau}f,g}_{2,\Omega}=-\bigrbr{\op S\rbr{\tau}\op P_{1}u+\op T\rbr{\tau}U,\op P_{1}^{\bot}v^{+}}_{2,\Omega}-\bigrbr{\op S\rbr{\tau}U+\rbr{\op T\rbr{\tau}^{+}}^{*}U-\varepsilon^{2}\mu U,v^{+}}_{2,\Omega}.
\]
Then, using (\ref{eq: Relation for A(=0003C4) and A(=0003C4)=00207A})
and the fact that $\op S\rbr{\tau}^{+}$ is the formal adjoint of~$\op S\rbr{\tau}$,
we get~(\ref{eq: Identity for U(=0003C4)}).

Another important identity is
\begin{equation}
\begin{aligned}\hspace{1.795em} & \hspace{-1.795em}\rbr{\op A_{\mu}\rbr{\tau}}^{-1}\op P_{1}-\rbr{\op A_{\mu}^{0}\rbr{\tau}}^{-1}\op P_{1}-\op K_{\mu}\rbr{\tau}+\op L_{\mu}\rbr{\tau}\op P_{1}+\op P_{1}\bigrbr{\op L_{\mu}\rbr{\tau}^{+}}^{*}\\
 & =-\bigrbr{\rbr{\op A_{\mu}\rbr{\tau}^{+}}^{-1}-\op K_{\mu}\rbr{\tau}^{+}}^{*}\op P_{1}^{\bot}\bigrbr{\op S\rbr{\tau}\rbr{\op A_{\mu}^{0}\rbr{\tau}}^{-1}\op P_{1}+\op T\rbr{\tau}\op K_{\mu}\rbr{\tau}}\\
 & \quad-\bigrbr{\rbr{\op A_{\mu}\rbr{\tau}^{+}}^{-1}-\rbr{\op A_{\mu}^{0}\rbr{\tau}^{+}}^{-1}\op P_{1}-\op K_{\mu}\rbr{\tau}^{+}}^{*}\bigrbr{\op S\rbr{\tau}^{+}+\op T\rbr{\tau}^{+}}^{*}\op K_{\mu}\rbr{\tau}\\
 & \quad-\bigrbr{\op S\rbr{\tau}^{+}\op K_{\mu}\rbr{\tau}^{+}-\varepsilon^{2}\bar{\mu}\rbr{\op A_{\mu}\rbr{\tau}^{+}}^{-1}}^{*}\op K_{\mu}\rbr{\tau}.
\end{aligned}
\label{eq: Identity for V(=0003C4)}
\end{equation}
To~prove this, we just note that $\op T\rbr{\tau}^{+}\op P_{1}=0$
and~$\op P_{1}\op K_{\mu}\rbr{\tau}=0$ and then apply~(\ref{eq: Identity for U(=0003C4)}).
We denote the operator on the left by~$\op V_{\mu}\rbr{\tau}$.

With~these results in hand, it is easy to complete the proofs of
the theorems.
\begin{proof}[Proof of Theorem~\textup{\ref{thm: Convergence on =0003A9}}]
We write
\[
\rbr{\op A_{\mu}\rbr{\tau}}^{-1}-\rbr{\op A_{\mu}^{0}\rbr{\tau}}^{-1}=\op U_{\mu}\rbr{\tau}+\rbr{\op A_{\mu}\rbr{\tau}}^{-1}\op P_{1}^{\bot}-\rbr{\op A_{\mu}^{0}\rbr{\tau}}^{-1}\op P_{1}^{\bot}+\op K_{\mu}\rbr{\tau}.
\]
By~Poincar\'{e}'s inequality~(\ref{est: Poincar=0000E9's inequality})
and Lemmas~\ref{lem: Estimates for A(=0003C4)} and~\ref{lem: Estimates for A=002070(=0003C4)}
(for~$\op A_{\mu}\rbr{\tau}^{+}$ and~$\op A_{\mu}^{0}\rbr{\tau}^{+}$,
respectively), the norms of $\abs{\tau}\rbr{\op A_{\mu}\rbr{\tau}}^{-1}\op P_{1}^{\bot}$,
$\op D_{2}\rbr{\tau}\rbr{\op A_{\mu}\rbr{\tau}}^{-1}\op P_{1}^{\bot}$
as well as $\abs{\tau}\rbr{\op A_{\mu}^{0}\rbr{\tau}}^{-1}\op P_{1}^{\bot}$,
$\op D_{2}\rbr{\tau}\rbr{\op A_{\mu}^{0}\rbr{\tau}}^{-1}\op P_{1}^{\bot}$
are uniformly bounded. In~Lemma~\ref{lem: Estimates for K(=0003C4)},
we proved that so are the norms of~$\abs{\tau}\op K_{\mu}\rbr{\tau}$
and~$\op D_{2}\rbr{\tau}\op K_{\mu}\rbr{\tau}$. Thus, it is  enough
to show that
\begin{align}
\norm{\op U_{\mu}\rbr{\tau}}_{\B\rbr{L_{2}\rbr{\Omega}}} & \lesssim\abs{\tau}^{-1},\label{est: U(=0003C4)}\\
\norm{\op D_{2}\rbr{\tau}\op U_{\mu}\rbr{\tau}}_{\B\rbr{L_{2}\rbr{\Omega}}} & \lesssim1.\label{est: D=002082(=0003C4)U(=0003C4)}
\end{align}
Let notation be as above. We use (\ref{est: S(=0003C4)}) and~(\ref{est: T(=0003C4)})
together with the Poincar\'{e} inequality~(\ref{est: Poincar=0000E9's inequality})
to estimate each term in~(\ref{eq: Identity for U(=0003C4)}). The~result
is that 
\begin{equation}
\begin{aligned}\abs{\rbr{\op U_{\mu}\rbr{\tau}f,g}_{2,\Omega}} & \lesssim\abs{\tau}^{-1}\bigrbr{\abs{\tau}\norm u_{1_{2},2,\Omega;\tau}+\norm{\op D_{1}U}_{1_{2},2,\Omega;\tau}}\\
 & \quad\quad\times\bigrbr{\norm{\op D_{1}\rbr{\tau}v^{+}}_{1_{2},2,\Omega;\tau}+\abs{\tau}\norm{v^{+}}_{1_{2},2,\Omega;\tau}}.
\end{aligned}
\label{est: (U(=0003C4)f,g)}
\end{equation}
Combining this with Lemmas~\ref{lem: Estimates for A(=0003C4)} (for~$\op A_{\mu}\rbr{\tau}^{+}$),
\ref{lem: Estimates for A=002070(=0003C4)} and~\ref{lem: Estimates for K(=0003C4)}
gives~(\ref{est: U(=0003C4)}).

The~inequality~(\ref{est: D=002082(=0003C4)U(=0003C4)}) is proved
in a like manner.\spacefactor2682{} We set $w^{+}=\rbr{\op A_{\mu}\rbr{\tau}^{+}}^{-1}\op D_{2}\rbr{\tau}^{*}g$
with $g\in L_{2}\rbr{\Omega}^{d}$ such that $\op D_{2}\rbr{\tau}^{*}g\in L_{2}\rbr{\Omega}$
and then estimate the form
\begin{equation}
\begin{aligned}\rbr{\op U_{\mu}\rbr{\tau}f,\op D_{2}\rbr{\tau}^{*}g}_{2,\Omega} & =-\bigrbr{\op S\rbr{\tau}\op P_{1}u+\op T\rbr{\tau}U,\op P_{1}^{\bot}w^{+}}_{2,\Omega}\\
 & \quad-\bigrbr{U,\rbr{\op S\rbr{\tau}^{+}+\op T\rbr{\tau}^{+}-\varepsilon^{2}\bar{\mu}}w^{+}}_{2,\Omega}.
\end{aligned}
\label{eq: (U(=0003C4)f,D=002082(=0003C4)*g)}
\end{equation}
However, a modification is required to eliminate the mixed second
derivatives of $w^{+}$ which arise when we estimate the right-hand
side (cf.~(\ref{est: (U(=0003C4)f,g)}), where a similar term, namely
$\op D_{2}\rbr{\tau}\op D_{1}\rbr{\tau}v^{+}$, causes no difficulty).

We do so as follows. Let $\varphi\in L_{2}\rbr{\Omega_{2};\widetilde H^{1}\rbr{\Omega_{1}}}$,
and let $\psi$ be the solution of
\[
\op D_{2}^{*}\rbr{\tau}\op D_{2}\rbr{\tau}\psi+\abs{\tau}^{2}\psi=\varphi
\]
in~$L_{2}\rbr{\Omega_{1};H^{1}\rbr{\Omega_{2}}}$. Obviously, $\psi$
has first derivatives and mixed second derivatives, as well as pure
second derivatives in~$x_{2}$. Fix $\vect l\in\R^{d_{1}}\oplus\cbr 0$
with $\abs{\vect l}=\abs{\tau}$ and define the operator~$\op E\rbr{\tau}\colon L_{2}\rbr{\Omega_{2};\widetilde H^{1}\rbr{\Omega_{1}}}\to\widetilde H^{1}\rbr{\Omega}^{d}$
that assigns to each $\varphi\in L_{2}\rbr{\Omega_{2};\widetilde H^{1}\rbr{\Omega_{1}}}$
the function~$\rbr{\vect l+\op D_{2}\rbr{\tau}}\psi$. It~follows
that $\rbr{\vect l+\op D_{2}\rbr{\tau}}^{*}\op E\rbr{\tau}$ is the
identity mapping. A~straightforward calculation (using the fact
that $\norm{\op E\rbr{\tau}\varphi}_{2,\Omega}=\norm{\psi}_{1_{2},2,\Omega;\tau}$)
shows that $\op E\rbr{\tau}$ is bounded and
\begin{equation}
\abs{\tau}\norm{\op D_{1}\rbr{\tau}\op E\rbr{\tau}\varphi}_{2,\Omega}+\norm{\op E\rbr{\tau}\varphi}_{1_{2},2,\Omega;\tau}\le\norm{\op D_{1}\rbr{\tau}\varphi}_{2,\Omega}+3\norm{\varphi}_{2,\Omega}.\label{est: E(=0003C4)}
\end{equation}

Now, we may rewrite the first expression on the right-hand side of~(\ref{eq: (U(=0003C4)f,D=002082(=0003C4)*g)})
as 
\[
\begin{aligned}\hspace{2em} & \hspace{-2em}\bigrbr{\op S\rbr{\tau}\op P_{1}u+\op T\rbr{\tau}U,\op P_{1}^{\bot}w^{+}}_{2,\Omega}\\
 & =\varepsilon\bigrbr{\sbr{\op D_{2},\op S\rbr{\tau}}\op P_{1}u+\sbr{\op D_{2},\op T\rbr{\tau}}U,\op E\rbr{\tau}\op P_{1}^{\bot}w^{+}}_{2,\Omega}\\
 & \quad+\bigrbr{\op S\rbr{\tau}\rbr{\vect l+\op D_{2}\rbr{\tau}}\op P_{1}u+\op T\rbr{\tau}\rbr{\vect l+\op D_{2}\rbr{\tau}}U,\op E\rbr{\tau}\op P_{1}^{\bot}w^{+}}_{2,\Omega}.
\end{aligned}
\]
Applying~(\ref{est: S(=0003C4)}) and~(\ref{est: T(=0003C4)}) and
similar results for the commutators of $\op D_{2}$ with~$\op S\rbr{\tau}$
and~$\op T\rbr{\tau}$ (notice that these commutators have the same
forms as $\op S\rbr{\tau}$ and~$\op T\rbr{\tau}$), we conclude
that
\[
\begin{aligned}\hspace{2em} & \hspace{-2em}\bigabs{\bigrbr{\op S\rbr{\tau}\op P_{1}u+\op T\rbr{\tau}U,\op P_{1}^{\bot}w^{+}}_{2,\Omega}}\\
 & \lesssim\bigrbr{\norm{\op D_{2}\rbr{\tau}u}_{1_{2},2,\Omega;\tau}+\abs{\tau}\norm u_{1_{2},2,\Omega;\tau}+\norm{\op D_{1}U}_{1_{2},2,\Omega;\tau}}\\
 & \quad\quad\times\bigrbr{\abs{\tau}\norm{\op D_{1}\rbr{\tau}\op E\rbr{\tau}\op P_{1}^{\bot}w^{+}}_{2,\Omega}+\norm{\op E\rbr{\tau}\op P_{1}^{\bot}w^{+}}_{1_{2},2,\Omega;\tau}}.
\end{aligned}
\]
Therefore, by the estimate~(\ref{est: E(=0003C4)}) together with
the Poincar\'{e} inequality,
\begin{equation}
\begin{aligned}\hspace{2em} & \hspace{-2em}\bigabs{\bigrbr{\op S\rbr{\tau}\op P_{1}u+\op T\rbr{\tau}U,\op P_{1}^{\bot}w^{+}}_{2,\Omega}}\\
 & \lesssim\bigrbr{\norm{\op D_{2}\rbr{\tau}u}_{1_{2},2,\Omega;\tau}+\abs{\tau}\norm u_{1_{2},2,\Omega;\tau}+\norm{\op D_{1}U}_{1_{2},2,\Omega;\tau}}\norm{\op D_{1}\rbr{\tau}w^{+}}_{2,\Omega}.
\end{aligned}
\label{est: (S(=0003C4)P=002081u+T(=0003C4)U,(I-P=002081)v=00207A), improved}
\end{equation}
The~second expression is handled in the same way as before. The~upshot
is that
\begin{equation}
\begin{aligned}\hspace{2em} & \hspace{-2em}\abs{\rbr{\op U_{\mu}\rbr{\tau}f,\op D_{2}\rbr{\tau}^{*}g}_{2,\Omega}}\\
 & \lesssim\bigrbr{\norm{\op D_{2}\rbr{\tau}u}_{1_{2},2,\Omega;\tau}+\abs{\tau}\norm u_{1_{2},2,\Omega;\tau}+\norm{\op D_{1}U}_{1_{2},2,\Omega;\tau}}\norm{w^{+}}_{1,2,\Omega;\tau},
\end{aligned}
\label{est: (U(=0003C4)f,D=002082(=0003C4)*g)}
\end{equation}
whence (\ref{est: D=002082(=0003C4)U(=0003C4)}) follows by Lemmas~\ref{lem: Estimates for A(=0003C4)}
(for~$\op A_{\mu}\rbr{\tau}^{+}$), \ref{lem: Estimates for A=002070(=0003C4)}
and~\ref{lem: Estimates for K(=0003C4)}.\qedspace
\end{proof}

\begin{proof}[Proof of Theorem~\textup{\ref{thm: Approximation on =0003A9 with 1st corrector}}]
The~proof follows the same pattern as the previous one. Again,
the assertion is reduced, by Poincar\'{e}'s inequality and Lemmas~\ref{lem: Estimates for A(=0003C4)}
and~\ref{lem: Estimates for A=002070(=0003C4)}, to the estimation
of~$\op D_{1}\rbr{\tau}\op U_{\mu}\rbr{\tau}$. Then the arguments
that we used to obtain (\ref{est: (U(=0003C4)f,D=002082(=0003C4)*g)})
go through without change to yield the desired conclusion.\qedspace
\end{proof}

\begin{proof}[Proof of Theorem~\textup{\ref{thm: Approximation on =0003A9 with 2d corrector}}]
We write
\[
\begin{aligned}\hspace{2em} & \hspace{-2em}\rbr{\op A_{\mu}\rbr{\tau}}^{-1}-\rbr{\op A_{\mu}^{0}\rbr{\tau}}^{-1}-\bigrbr{\op K_{\mu}\rbr{\tau}-\op L_{\mu}\rbr{\tau}}\op P_{1}-\op P_{1}\bigrbr{\op K_{\mu}\rbr{\tau}^{+}-\op L_{\mu}\rbr{\tau}^{+}}^{*}\\
 & =\op V_{\mu}\rbr{\tau}+\bigrbr{\op P_{1}^{\bot}\rbr{\op A_{\mu}\rbr{\tau}^{+}}^{-1}-\op P_{1}^{\bot}\rbr{\op A_{\mu}^{0}\rbr{\tau}^{+}}^{-1}-\op K_{\mu}\rbr{\tau}^{+}}^{*}
\end{aligned}
\]
(recall that $\op V_{\mu}\rbr{\tau}$ is the operator on the left
side of~(\ref{eq: Identity for V(=0003C4)})). Since Theorem~\ref{thm: Approximation on =0003A9 with 1st corrector}
holds true for $\op A_{\mu}\rbr{\tau}^{+}$, it follows that
\[
\norm{\op P_{1}^{\bot}\rbr{\op A_{\mu}\rbr{\tau}^{+}}^{-1}-\op P_{1}^{\bot}\rbr{\op A_{\mu}^{0}\rbr{\tau}^{+}}^{-1}-\op K_{\mu}\rbr{\tau}^{+}}_{\B\rbr{L_{2}\rbr{\Omega}}}\lesssim1,
\]
where we have used Poincar\'{e}'s inequality and the fact that~$\op P_{1}\op K_{\mu}\rbr{\tau}^{+}=0$.
Thus, we are left with estimating the operator~$\op V_{\mu}\rbr{\tau}$.

Fix~$f,g\in L_{2}\rbr{\Omega}$.  Let $u$, $U$ and~$v^{+}$
be as above, and let $u^{+}=\rbr{\op A_{\mu}^{0}\rbr{\tau}^{+}}^{-1}g$
and~$U^{+}=\op K_{\mu}\rbr{\tau}^{+}g$. Then, by~(\ref{eq: Identity for V(=0003C4)}),
\[
\begin{aligned}\rbr{\op V_{\mu}\rbr{\tau}f,g}_{2,\Omega} & =-\bigrbr{\op S\rbr{\tau}\op P_{1}u+\op T\rbr{\tau}U,\op P_{1}^{\bot}\rbr{v^{+}-U^{+}}}_{2,\Omega}\\
 & \quad-\bigrbr{U,\rbr{\op S\rbr{\tau}^{+}+\op T\rbr{\tau}^{+}}\rbr{v^{+}-\op P_{1}u^{+}-U^{+}}}_{2,\Omega}\\
 & \quad-\bigrbr{U,\op S\rbr{\tau}^{+}U^{+}-\varepsilon^{2}\bar{\mu}v^{+}}_{2,\Omega}.
\end{aligned}
\]
We use (\ref{est: (S(=0003C4)P=002081u+T(=0003C4)U,(I-P=002081)v=00207A), improved}),
with $v^{+}-U^{+}$ in place of~$w^{+}$, to estimate, dropping
the constants, the first expression on the right-hand side~by
\begin{equation}
\bigrbr{\norm{\op D_{2}\rbr{\tau}u}_{1_{2},2,\Omega;\tau}+\abs{\tau}\norm u_{1_{2},2,\Omega;\tau}+\norm{\op D_{1}U}_{1_{2},2,\Omega;\tau}}\norm{\op D_{1}\rbr{\tau}\op P_{1}^{\bot}\rbr{v^{+}-U^{+}}}_{2,\Omega}.\label{est: V(=0003C4), first expression}
\end{equation}
The~remaining terms,  according to estimates similar to (\ref{est: S(=0003C4)})
and~(\ref{est: T(=0003C4)}) as well as Poincar\'{e}'s inequality,
do not exceed
\begin{equation}
\norm{\op D_{1}U}_{1_{2},2,\Omega;\tau}\bigrbr{\norm{v^{+}-\op P_{1}u^{+}-U^{+}}_{1,2,\Omega;\tau}+\norm{\op D_{1}U^{+}}_{1_{2},2,\Omega;\tau}+\abs{\tau}\norm{\op D_{1}\rbr{\tau}v^{+}}_{2,\Omega}}\label{est: V(=0003C4), remaining terms}
\end{equation}
(notice here that~$\op P_{1}\op K_{\mu}\rbr{\tau}=0$). Further,
by Poincar\'{e}'s inequality,
\begin{equation}
\norm{\op D_{1}\rbr{\tau}\op P_{1}^{\bot}\rbr{v^{+}-U^{+}}}_{2,\Omega}\lesssim\norm{\op D_{1}\rbr{\tau}\rbr{v^{+}-u^{+}-U^{+}}}_{2,\Omega}+\norm{\op D_{1}\rbr{\tau}\op D_{1}\rbr{\tau}u^{+}}_{2,\Omega}\label{est: V(=0003C4), first expression, additional}
\end{equation}
and
\begin{equation}
\begin{aligned}\norm{v^{+}-\op P_{1}u^{+}-U^{+}}_{1,2,\Omega;\tau} & \lesssim\norm{\op D_{1}\rbr{\tau}\rbr{v^{+}-u^{+}-U^{+}}}_{2,\Omega}+\norm{v^{+}-u^{+}}_{1_{2},2,\Omega;\tau}\\
 & \quad+\norm{\op D_{1}\rbr{\tau}u^{+}}_{1,2,\Omega;\tau}+\norm{\op D_{1}U^{+}}_{1_{2},2,\Omega;\tau}.
\end{aligned}
\label{est: V(=0003C4), remaining terms, additional}
\end{equation}
If~we combine (\ref{est: V(=0003C4), first expression}) with (\ref{est: V(=0003C4), first expression, additional})
and~(\ref{est: V(=0003C4), remaining terms}) with (\ref{est: V(=0003C4), remaining terms, additional})
and apply Lemmas~\ref{lem: Estimates for A=002070(=0003C4)} and~\ref{lem: Estimates for K(=0003C4)}
(for~$\op K_{\mu}\rbr{\tau}$ and~$\op K_{\mu}\rbr{\tau}^{+}$)
and Theorems~\ref{thm: Convergence on =0003A9} and~\ref{thm: Approximation on =0003A9 with 1st corrector}
(for~$\op A_{\mu}\rbr{\tau}^{+}$), then we obtain
\[
\abs{\rbr{\op V_{\mu}\rbr{\tau}f,g}_{2,\Omega}}\lesssim\norm f_{2,\Omega}\norm g_{2,\Omega}.
\]
This proves the theorem.\qedspace
\end{proof}

\subsection*{Acknowledgment}
The~author is grateful to \textsc{T.\thinglue A.~Suslina}
for useful conversations about the problem.
He also would like to thank \textsc{V.\thinglue P.~Smyshlyaev} for his interest
in this work.

\bibliographystyle{alpha}
\bibliography{bibliography}

\begin{thebibliography}{\textsc{ZhKO}}

\bibitem[\textsc{Al}]{Al}
\textsc{G.~Allaire},
\newblock \textit{Homogenization and two-scale convergence},
\newblock SIAM J. Math. Anal., 23 (1992), pp.~1482--1518.

\bibitem[\textsc{BP}]{BP}
\textsc{N.~Bakhvalov and G.~Panasenko},
\newblock \textit{Homogenisation: Averaging Processes in Periodic Media:
                  Mathematical Problems in the Mechanics of Composite Materials},
\newblock Nauka, Moscow, 1984 (in Russian);
\newblock Kluwer Academic, Dordrecht, 1989 (in English).

\bibitem[\textsc{BLP}]{BLP}
\textsc{A.~Bensoussan, J.-L.~Lions and G.~Papanicolaou},
\newblock \textit{Asymptotic Analysis for Periodic Structures},
\newblock North-Holland, Amsterdam, 1978.

\bibitem[\textsc{BSu1}]{BSu1}
\textsc{M.~Birman and T.~Suslina},
\newblock \textit{Threshold effects near the lower edge of the spectrum
                  for periodic differential operators of mathematical physics},
\newblock in Systems, Approximation, Singular Integral Operators, and Related Topics,
          Alexander~A.~Borichev and Nikolai~K.~Nikolski, eds.,
\newblock Birkh\"auser, Basel, 2001, pp.~71--107.

\bibitem[\textsc{BSu2}]{BSu2}
\textsc{M.~Sh.~Birman and T.~A.~Suslina},
\newblock \textit{Second order periodic differential operators. Threshold properties and homogenization},
\newblock Algebra i~Analiz, 15 (2003), no.~5, pp.~1--108 (in Russian);
\newblock St.~Petersburg Math.~J., 15 (2004), pp.~639--714 (in English).

\bibitem[\textsc{BSu3}]{BSu3}
\textsc{M.~Sh.~Birman and T.~A.~Suslina},
\newblock \textit{Homogenization with corrector term for periodic elliptic differential operators},
\newblock Algebra i~Analiz, 17 (2005), no.~6, pp.~1--104 (in Russian);
\newblock St.~Petersburg Math.~J., 17 (2006), pp.~897--973 (in English).

\bibitem[\textsc{BSu4}]{BSu4}
\textsc{M.~Sh.~Birman and T.~A.~Suslina},
\newblock \textit{Homogenization with corrector for periodic differential operators.
                  Approximation of solutions in the Sobolev class~$H^{1}(\mathbb{R}^{d})$},
\newblock Algebra i~Analiz, 18 (2006), no.~6, pp.~1--130 (in Russian);
\newblock St.~Petersburg Math.~J., 18 (2007), pp.~857--955 (in English).

\bibitem[\textsc{Bo}]{Bo}
\textsc{D.~I.~Borisov},
\newblock \textit{Asymptotics for the solutions of elliptic systems with rapidly oscillating coefficients},
\newblock Algebra i~Analiz, 20 (2008), no.~2, pp.~19--42 (in Russian);
\newblock St.~Petersburg Math.~J., 20 (2009), pp.~175--191 (in English).

\bibitem[\textsc{BCSu}]{BCSu}
\textsc{R.~Bunoiu, G.~Cardone and T.~Suslina},
\newblock \textit{Spectral approach to homogenization of an elliptic operator
                  periodic in some directions},
\newblock Math. Meth. Appl. Sci., 34 (2011), pp.~1075--1096.

\bibitem[\textsc{ChC}]{ChC}
\textsc{K.~D.~Cherednichenko and S.~Cooper},
\newblock \textit{Resolvent estimates for high-contrast elliptic problems with periodic coefficients},
\newblock Arch. Ration. Mech. Anal., to appear.

\bibitem[\textsc{Gr1}]{Gr1}
\textsc{G.~Griso},
\newblock \textit{Error estimate and unfolding for periodic homogenization},
\newblock Asymptot. Anal., 40 (2004), pp.~269--286.

\bibitem[\textsc{Gr2}]{Gr2}
\textsc{G.~Griso},
\newblock \textit{Interior error estimate for periodic homogenization},
\newblock Anal. Appl., 4 (2006), pp.~61--79.

\bibitem[\textsc{KLS}]{KLS}
\textsc{C.~E.~Kenig, F.~Lin and Z.~Shen},
\newblock \textit{Convergence Rates in $L_{2}$ for elliptic homogenization problems},
\newblock Arch. Ration. Mech. Anal., 203 (2012), pp.~1009--1036.

\bibitem[\textsc{LNW}]{LNW}
\textsc{D.~Lukkassen, G.~Nguetseng and P.~Wall},
\newblock \textit{Two-scale convergence},
\newblock Int. J. Pure Appl. Math., 2 (2002), pp.~35--86.

\bibitem[\textsc{MSh}]{MSh}
\textsc{V.~G.~Maz'ya and T.~O.~Shaposhnikova},
\newblock \textit{Theory of Sobolev Multipliers: With Applications to Differential and Integral Operators},
\newblock Springer, Berlin, 2009.

\bibitem[\textsc{OShY}]{OShY}
\textsc{O.~A.~Oleinik, A.~S.~Shamaev and G.~A.~Yosifian},
\newblock \textit{Mathematical Problems in Elasticity and Homogenization},
\newblock North-Holland, Amsterdam, 1992.

\bibitem[\textsc{PSu}]{PSu}
\textsc{M.~A.~Pakhnin and T.~A.~Suslina},
\newblock \textit{Operator error estimates for homogenization
                  of the elliptic Dirichlet problem in a bounded domain},
\newblock Algebra i~Analiz, 24 (2012), no.~6, pp.~139--177 (in Russian);
\newblock St.~Petersburg Math.~J., 24 (2013), pp.~949--976 (in English).

\bibitem[\textsc{P}]{P}
\textsc{S.~E.~Pastukhova},
\newblock \textit{Approximations of the resolvent for a non-self-adjoint
                  diffusion operator with rapidly oscillating coefficients},
\newblock Mat. Zametki, 94 (2013), pp.~130--150 (in Russian);
\newblock Math. Notes, 94 (2013), pp.~127--145 (in English).

\bibitem[\textsc{PT}]{PT}
\textsc{S.~E.~Pastukhova and R.~N.~Tikhomirov},
\newblock \textit{Operator estimates in reiterated and locally periodic homogenization},
\newblock Dokl. Acad. Nauk, 415 (2007), pp.~304--309 (in Russian);
\newblock Dokl. Math., 76 (2007), pp.~548--553 (in English).

\bibitem[\textsc{Se1}]{Se1}
\textsc{N.~N.~Senik},
\newblock \textit{Homogenization for a periodic elliptic operator
                  in a strip with various boundary conditions},
\newblock Algebra i~Analiz, 25 (2013), no.~4, pp.~182--259 (in Russian);
\newblock St.~Petersburg Math.~J., 25 (2014), pp.~647--697 (in English).

\bibitem[\textsc{Se2}]{Se2}
\textsc{N.~N.~Senik},
\newblock \textit{On homogenization for non-self-adjoint periodic elliptic
                  operators on an infinite cylinder},
\newblock Funktsional. Anal. i Prilozhen., 50 (2016), no.~1, to appear (in Russian).

\bibitem[\textsc{Su1}]{Su1}
\textsc{T.~A.~Suslina},
\newblock \textit{On homogenization for a periodic elliptic operator in a strip},
\newblock Algebra i~Analiz, 16 (2004), no.~1, pp.~269--292 (in Russian);
\newblock St.~Petersburg Math.~J., 16 (2005), pp.~237--257 (in English).

\bibitem[\textsc{Su2}]{Su2}
\textsc{T.~A.~Suslina},
\newblock \textit{Homogenization of a stationary periodic Maxwell system},
\newblock Algebra i~Analiz, 16 (2004), no.~5, pp.~162--244 (in Russian);
\newblock St.~Petersburg Math.~J., 16 (2005), pp.~863--922 (in English).

\bibitem[\textsc{Su3}]{Su3}
\textsc{T.~A.~Suslina},
\newblock \textit{Homogenization in the Sobolev class~$H^{1}(\mathbb{R}^{d})$
                  for second order periodic elliptic operators
                  with the inclusion of first order terms},
\newblock Algebra i~Analiz, 22 (2010), no.~1, pp.~108--222 (in Russian);
\newblock St.~Petersburg Math.~J., 22 (2011), pp.~81--162 (in English).

\bibitem[\textsc{Su4}]{Su4}
\textsc{T.~A.~Suslina},
\newblock \textit{Homogenization of elliptic systems with periodic
                  coefficients: operator error estimates in~$L_2(\mathbb R^d)$
                  with corrector taken into account},
\newblock Algebra i~Analiz, 26 (2014), no.~4, pp.~195--263 (in Russian);
\newblock St.~Petersburg Math.~J., 26 (2015), pp.~643--693 (in English).

\bibitem[\textsc{S-HT}]{S-HT}
\textsc{J.~Sanchez-Hubert and N.~Turbe},
\newblock \textit{Ondes \'{e}lastiques dans une bande p\'{e}riodique},
\newblock Mod\'{e}l. Math. Anal. Num\'{e}r., 20 (1986), pp.~539--561.

\bibitem[\textsc{ZhKO}]{ZhKO}
\textsc{V.~V.~Zhikov, S.~M.~Kozlov and O.~A.~Oleinik},
\newblock \textit{Homogenization of differential operators and integral functionals},
\newblock Nauka, Moscow, 1993 (in Russian);
\newblock Springer, Berlin, 1994 (in English).

\bibitem[\textsc{Zh}]{Zh}
\textsc{V.~V.~Zhikov},
\newblock \textit{On operator estimates in homogenization theory},
\newblock Dokl. Acad. Nauk, 403 (2005), pp.~305--308 (in Russian);
\newblock Dokl. Math., 72 (2005), pp.~535--538 (in English).

\bibitem[\textsc{ZhP}]{ZhP}
\textsc{V.~V.~Zhikov and S.~E.~Pastukhova},
\newblock \textit{On operator estimates for some problems in homogenization theory},
\newblock Russ.~J. Math. Phys., 12 (2005), pp.~515--524.

\end{thebibliography}

\end{document}